\documentclass[11pt,a4paper]{amsart}
\usepackage{amsmath}
\usepackage{amsfonts}
\usepackage{graphicx}
\usepackage{framed}
\usepackage{amssymb,cancel}
\usepackage{xcolor}
\usepackage{esint}

\usepackage{etex}
\usepackage{latexsym}
\usepackage[english]{babel}

\def \N {\mathbb{N}}
\def \R {\mathbb{R}}

\def \de {\partial}

\def \e {\varepsilon}

\def \LL {\mathcal{L}_{\varepsilon}}

\def \Oo {\mathcal{O}}

\usepackage{amsthm}
\theoremstyle{definition}
\newtheorem{definition}{Definition}[section]

\newtheorem{remark}[definition]{Remark}

\theoremstyle{plain}
\newtheorem{theorem}[definition]{Theorem}
\newtheorem{proposition}[definition]{Proposition}
\newtheorem{lemma}[definition]{Lemma}
\newtheorem{corollary}[definition]{Corollary}

\allowdisplaybreaks

\numberwithin{equation}{section}
\usepackage[colorlinks=true,urlcolor=blue,
citecolor=red,linkcolor=blue,linktocpage,pdfpagelabels,
bookmarksnumbered,bookmarksopen]{hyperref}
\usepackage[hyperpageref]{backref}

\makeatletter
\@namedef{subjclassname@2020}{\textup{2020} Mathematics Subject Classification}
\makeatother

\begin{document}

 \title[Mixed local-nonlocal concave-convex critical problems]
 {On the existence of a second positive solution to mixed local-nonlocal concave-convex critical problems}

 \author[S.\,Biagi]{Stefano Biagi}
 \author[E.\,Vecchi]{Eugenio Vecchi}
 
 \address[S.\,Biagi]{Dipartimento di Matematica
 \newline\indent Politecnico di Milano \newline\indent
 Via Bonardi 9, 20133 Milano, Italy}
 \email{stefano.biagi@polimi.it}
 
 \address[E.\,Vecchi]{Dipartimento di Matematica
 \newline\indent Universit\`a di Bologna \newline\indent
 Piazza di Porta San Donato 5, 40126 Bologna , Italy}
 \email{eugenio.vecchi2@unibo.it}



\thanks{The Authors are
member of the {\em Gruppo Nazionale per
l'Analisi Ma\-te\-ma\-ti\-ca, la Probabilit\`a e le loro Applicazioni}
(GNAMPA) of the {\em Istituto Nazionale di Alta Matematica} (INdAM), and are
partially 
supported by the PRIN 2022 project 2022R537CS \emph{$NO^3$ - Nodal Optimization, NOnlinear elliptic equations, NOnlocal geometric problems, with a focus on regularity}, founded by the European Union - Next Generation EU}

 \begin{abstract}
 We prove the existence of a second positive weak solution for mixed local-nonlocal critical semilinear elliptic problems  with a sublinear perturbation in the spirit of \cite{ABC}.
 \end{abstract}
 \maketitle 
\section{Introduction}\label{sec.Intro}
Let $\Omega \subset \mathbb{R}^n$ be an open and bounded set with smooth enough boundary $\partial \Omega$. We consider the following mixed local-nonlocal perturbed critical semilinear elliptic problem:
\begin{equation}\tag{{$\mathrm{P}_{\varepsilon}$}}\label{eq:Problem}
\left\{ \begin{array}{rl}
\LL u  = \lambda \,u^{p} + u^{2^{\ast}-1} & \textrm{in } \Omega,\\
u>0 & \textrm{in } \Omega,\\
u= 0 & \textrm{in } \mathbb{R}^{n}\setminus \Omega,
\end{array}\right.
\end{equation}
 \noindent where $\lambda >0$ is a positive real parameter, $p \in (0,1)$ and
 $$\LL  := -\Delta  +\varepsilon\, (-\Delta)^s, \qquad s\in (0,1), \quad \varepsilon \in (0,1].$$ 
  Along the paper it will sometimes be useful to denote the above problem as \eqref{eq:Problem}$_\lambda$ 
   to make it clear the choice of the parameter.
Here, $(-\Delta)^s$ with $s\in (0,1)$ denotes the fractional Laplacian which acts in smooth enough functions as
\begin{equation*}
 (-\Delta)^s u(x) = 2\,\mathrm{P.V.}\int_{\R^n}\frac{u(x)-u(y)}{|x-y|^{n+2s}}\,dy = 2\,\lim_{\delta\to 0^+}\int_{\{|x-y|\geq\delta\}}\frac{u(x)-u(y)}{|x-y|^{n+2s}}\,dy.
\end{equation*}
We neglect the normalization constant $C_{n,s} > 0$ usually appearing in front of the integral because we are not interested in asymptotics as $s \to 1^{-}$.

\medskip

The above boundary value problem falls in the framework of the so called {\it concave-convex problems} whose model and most famous example is the one considered in the paper \cite{ABC}:
\begin{equation}\tag{ABC}\label{eq:ABC}
\left\{ \begin{array}{rl}
-\Delta u  = \lambda \,u^{p} + u^{2^{\ast}-1} & \textrm{in } \Omega,\\
u>0 & \textrm{in } \Omega,\\
u= 0 & \textrm{on } \partial \Omega.
\end{array}\right.
\end{equation}
Without any aim of completeness, we  mention that \eqref{eq:ABC} has been generalized to various directions, either allowing for different operators either considering different boundary conditions, see e.g. \cite{AzMaPe, CoPe, CCP, BESS}.
\medskip

Before entering into the details related to \eqref{eq:Problem}$_{\lambda}$, let us spend a few more words on $\mathcal{L}_{\varepsilon}$.
The operator $\mathcal{L}_{1}$ (i.e. with $\varepsilon=1$) is a special instance of a wide more general class of operators whose study began in the '60s, see \cite{BCP} and \cite{Cancelier} for generalizations, in connection with the validity of a maximum principle. On the other hand, the operator $\mathcal{L}_{1}$ can be seen as the infinitesimal generator of a stochastic process obtained as a superposition of a Brownian motion and a L\'{e}vy flight, and hence there is a vast literature which establishes several regularity properties 
\emph{adopting probabilistic techniques}, see e.g. \cite{CKSV} and the references therein. \\
\noindent More recently, the study of regularity properties related to this operator (and its quasilinear generalizations) has seen an increasing interest, mainly adopting more analytical and PDEs approaches, see, e.g., 
\cite{BDVV, CDV22, DeFMin, DPLV2, GarainKinnunen, GarainLindgren, SVWZ, ArRa, SVWZ2}. 
It is worth mentioning that the operator $\LL$ (usually with $\varepsilon=1$) seems to be of interest in biological applications, see, e.g. \cite{DV} and the references therein.\\
\noindent The operator $\mathcal{L}_{1}$ has a also variational nature and it is associated to the energy 
\begin{equation*}
E(u):= \int_{\Omega}\dfrac{|\nabla u|^2}{2}\, dx + \dfrac{1}{2}\, \iint_{\mathbb{R}^{2n}}\dfrac{|u(x)-u(y)|^2}{|x-y|^{n+2s}}\, dx \, dy.
\end{equation*}
\noindent defined on a suitable space functions for which we refer to Section \ref{sec.Prel}.
It is clear that 
\begin{equation*}
E(\lambda u) = \lambda^2 E(u), \quad \textrm{for every } \lambda \in \mathbb{R},
\end{equation*}
\noindent but there is a lack of scaling invariance, namely,
\begin{equation*}
E(u_t) = \dfrac{1}{2}\int_{\Omega} |\nabla u|^2 + \dfrac{t^{2s-2}}{2} \iint_{\mathbb{R}^{2n}}\dfrac{|u(x)-u(y)|^2}{|x-y|^{n+2s}}\,dx \,dy \quad \textrm{for every } t>0,
\end{equation*}
 where $u_{t}(x):= t^{(n-2)/2}u(x)$, which typically preserves the $L^2$-gradient norm. 
 
 Among other aspects, this phenomenon is responsible of the following fact: the best Sobolev constant in the natural mixed Sobolev inequality is never achieved and it coincides with the one coming from the purely local one. We notice that something similar happens when dealing with mixed Hardy-type inequalities, see \cite{BEMV}. An interesting consequence of the absence of mixed Aubin-Talenti functions naturally affects PDEs with critical term. Indeed, when dealing with critical variational problems for $-\Delta$, it is well known that a major role is played by the best Sobolev constant, for example being a threshold of validity of Palais-Smale condition, and this is usually achieved by testing the critical equation with the Aubin-Talenti functions. An analogous procedure can be followed in the mixed setting as well, but the lack of invariance previously mentioned can create troubles, see e.g \cite{BDVV5, BV2, daSiFi}. We stress that something similar may happen with non-homogeneous operators like the $(p,q)$-Laplacian, see \cite{KRS}.

 Following the approach recently used in \cite{BV2} for the case of mild singular and critical mixed problems, we are now interested in the following classical problem:
\begin{center}
{\it Find values of the parameter $\lambda$ for which \eqref{eq:Problem}$_{\lambda}$ admits one or more positive weak solutions.}
\end{center}
We refer to Definition \ref{def:weaksol} for the precise definition of \emph{weak solution} of \eqref{eq:Problem}$_{\lambda}$.
Differently from \cite{BDVV5} and \cite{BV2}, here we are in presence of a sublinear perturbation ($\lambda \, u^{p}$ with $p\in (0,1)$) of the critical term ($u^{2^{\ast}-1}$).

\medskip

For $\varepsilon=1$, a first step towards an answer has been recently made in \cite[Theorem 1.2]{AMT} and \cite[Theorem 1.1]{MS} where the following has been proved:
\begin{theorem} \label{thm:ALTRI}
 Let $\Omega\subset\R^n$ \emph{(}with $n\geq 3$\emph{)} be a bounded open set
 with smooth enough boundary, and let $p \in (0,1)$. Then, there exists $\Lambda > 0$
 such that
 \begin{itemize}
 \item[a)]  problem $(\mathrm{P}_{1})_{\lambda}$ admits at least one weak solution for every $0<\lambda \leq \Lambda$;
  \item[b)]
    problem $(\mathrm{P}_{1})_{\lambda}$ does not admit weak solutions
  for every $\lambda>\Lambda$.
 \end{itemize}
 Moreover, for $\lambda \in (0,\Lambda)$, the solution is minimal and increasing w.r.t. to $\lambda$.
\end{theorem}
We stress that the the above results actually holds in the supercritical case as well.
For sake of completeness, we recall that the existence of the first positive solution in \cite{AMT} is obtained by means of a sub/su\-per\-so\-lu\-tion scheme as in \cite{ABC}. We also note that, if we fix $\varepsilon \in (0,1]$, the above result holds for a threshold $\Lambda_{\varepsilon}$ now depending on $\varepsilon$.\\
We also mention that concave-convex problems in the mixed local-nonlocal setting has been previously studied in \cite{daSiSa}: there, the leading operator is a combination of $p$-Laplacian and fractional $p$-Laplacian and the existence of weak solutions is proved under the assumption $p>n$, being the interest of the authors to study the limiting problem as $p\to +\infty$.\\

Our first result is somehow complementary to Theorem \ref{thm:ALTRI} stated above and it is the counterpart of \cite[Theorem 2.2]{ABC} in the mixed local-nonlocal setting.
 
\begin{theorem} \label{thm:main}
Let $\varepsilon \in (0,1]$ be fixed. Then, there exists a constant $M_{\varepsilon}>0$ such that, for every $\lambda \in (0,\Lambda_\varepsilon)$, problem \eqref{eq:Problem}$_{\lambda,\e}$ has at most one  solution $u_{\lambda}$ with 
  $$\|u_{\lambda,\e}\|_{L^{\infty}(\mathbb{R}^{n})} \leq M_{\varepsilon}.$$
\end{theorem} 
Note that the if the unique solution of Theorem \ref{thm:main} exists, it has to be the minimal one found in Theorem \ref{thm:ALTRI}.

\medskip

Following the seminal paper \cite{ABC}, it is natural to wonder whether a second positive solution exists or not for all $\lambda \in (0,\Lambda)$. A similar question has been answered in \cite{MS} for sublinear perturbations of a subcritical nonlinearity, so leaving the critical case unsolved. 

To better understand the difficulties one has to face in the case of mixed operators like $\LL$, let us briefly recall the method employed in \cite{ABC} to find a second positive solution of \eqref{eq:ABC}:
\begin{itemize}
\item[i)] show that the first solution found is a minimizer of the functional naturally associated with \eqref{eq:ABC} in $C^1$-topology;
\item[ii)] thanks to the famous result by Brezis and Nirenberg \cite{BNH1C1}, one inherits that the first solution is actually a minimizer in the $H^1$-topology;
\item[iii)] prove the validity of the Palais-Smale condition under a certain value proportional to the best Sobolev constant, here making use of the Aubin-Talenti functions;
\item[iv)] find a second positive solution of mountain-pass-type.
\end{itemize}

Despite i) can still be proved in our case, the above scheme finds a first obstacle once a result like the one in \cite{BNH1C1} is needed. To this aim, and thanks to the recent regularity results proved in \cite{AC2} and \cite{SVWZ2}, we establish the following

\begin{theorem} \label{thm:H1vsC1}
Let $\Phi: \mathcal{X}^{1,2}(\Omega) \to \mathbb{R}$ defined as 
\begin{equation}\label{eq.DefPhi}
\Phi(u):= \dfrac{1}{2}\,\rho(u)^2 - \int_{\Omega}F(x,u)\, dx,
\end{equation}
\noindent where $F(x,u):=\int_{0}^{u}f(x,s)\,ds$ and with $f$ satisfying that
\begin{equation}\label{eq:growth_on_f}
|f(x,u)|\leq C_f\, (1+|u|^{2^*-1}).
\end{equation}
Let $u_0 \in \mathcal{X}^{1,2}(\Omega)$ be a local minimizer of $\Phi$ in the $C^1$-topology. Then, $u_0$ is a local minimizer in the $\mathcal{X}^{1,2}$-topology.
\end{theorem}

Theorem \ref{thm:H1vsC1} allows to achieve ii) as well,
but a new difficulty arises once iii) has to be faced, the main reason 
being the lack of scaling invariance, and its consequences on the mixed Sobolev inequality, previously discussed. 
A similar issue occurred in the recent \cite{BV2} when dealing with mild singular problems, and it is the main 
reason to consider $\LL$ instead of $\mathcal{L}_{1}$ because, roughly speaking, the presence of the parameter $
\varepsilon$ allows us to adjust the lack of scaling invariance, at least for small enough $\varepsilon$. In this 
way, we are able to fully follow the above scheme proving the following
 
\begin{theorem} \label{thm:main2}
 Let $\Omega\subset\R^n$ \emph{(}with $n\geq 3$\emph{)} be a bounded open set
 with smooth enough boundary, and let $p\in (0,1)$ be fixed.
 
 Then, there exist $\lambda_{\star} > 0$ and $
 \varepsilon_0 \in (0,1)$
 such that problem \eqref{eq:Problem}$_{\lambda}$
 admits a second positive solutions for every $\varepsilon \in (0,\varepsilon_0)$ and for every $\lambda \in (0,
 \lambda_{\star})$.
\end{theorem}
As it clearly appears from the statement, there is a first price to pay once considering $\LL$, because the existence is not proved for every $\lambda$. Less evident is the second fee we have to pay. Let us briefly describe it: the first solution found in Theorem \ref{thm:ALTRI} can still be found once $\mathcal{L}_{1}$ is replaced with $\LL$ but now it implicitly depends on $\varepsilon$ and this forces to carefully keep track of the dependence on $\varepsilon$ of many ingredients like $L^{\infty}$-bounds and so on.\\
As in \cite{BV2}, the key result to be proved is Lemma \ref{lem:CrucialLemma} and we want to explicit mention that a careful inspection of its proof shows that a better result can be obtained, this time with big restrictions on both the dimension $n$ and the fractional parameter $s$:
\begin{corollary}
Let $n=3$ and $s\in \left(0,\tfrac{1}{2}\right)$. Then problem \eqref{eq:Problem}$_{\lambda}$ admits at least two positive weak solutions for every $\lambda \in (0,\Lambda)$.
\end{corollary}
Finally, once the existence of (at least) two solutions of problem  
\eqref{eq:Problem}$_{\lambda}$ has been established (at least for $\e$ small enough),
by exploiting Theorem \ref{thm:main} and by proceeding essentially
as in \cite{ABC} we can prove the following qualitative result.
\begin{theorem} \label{thm:MAIN3}
 Let $\e\in(0,1]$ be fixed, and assume that $\Omega$ is \emph{star-shaped}. Then, 
 $$\text{$\|v_{\lambda,\e}\|_{L^\infty(\Omega)}\to+\infty$ as $\lambda\to 0^+$},$$ where
 $v_{\lambda,\e}\in\mathcal{X}^{1,2}(\Omega)$ is \emph{any weak solution of}
 problem \eqref{eq:Problem}$_{\lambda}$ distinct from its mini\-mal solution
 $u_{\lambda,\e}$ \emph{(}see Theorem \ref{thm:ALTRI} and Lemma \ref{lem:existencemusharp}\emph{)}.
\end{theorem} 
 \medskip

\noindent \emph{Plan of the paper:}
 The paper is organized as follows: 
 \begin{itemize}
  \item In Section \ref{sec.Prel} we collect all the relevant notation, 
   definitions and preliminary results needed for the proof
   of our main results.
   \item In Section \ref{sec:sublinear} we briefly 
   study the \emph{purely sublinear} counterpart of problem
   \eqref{eq:Problem}$_{\lambda}$ (starting from the results
   proved in \cite{BMV}); in particular, we establish some
   \emph{uniform bounds} of the unique solution $w_{\lambda,\e}$ of this problem
   which will be used subsequent sections, and we establish
   Theorem \ref{thm:main}.
   \item In Section \ref{sec.H1C1} we prove Theorem \ref{thm:H1vsC1}.
   \item In Section \ref{sec:proofMain2} we prove Theorem \ref{thm:main2}.
   \item Finally, in Section \ref{sec:Last} we give the proof of Theorem \ref{thm:MAIN3}.
\end{itemize}
\section{Preliminaries}\label{sec.Prel}
 In this section we collect some preliminary definitions and results which
 will be used throughout the rest of the paper. 
 First of all, we review some basic properties
 of the fractional Sobolev spaces, and
 we properly
 introduce the adequate functional setting for the study
 of mixed local-nonlocal operators;  we then give the precise definition
 of \emph{weak sub/supersolution} of problem \eqref{eq:Problem}$_{\lambda}$,
 and we establish some qualitative properties of the solutions of this problem
 (provided they exist).
 Finally, we spend a few words the applicability
 of Theorem \ref{thm:ALTRI} (established in the case
  when $\e = 1$) to our $\e$-dependent operator $\LL$.
 \medskip
 
 \noindent\textbf{i)\,\,Sobolev spaces of fractional order.} 
 We begin this section by collecting a few basic facts
 fractional Sobolev spaces, which are naturally
 related to the fractional Laplacian $(-\Delta)^s$; we refer
 to \cite{LeoniFract} for a thorough introduction to this topic.
 \vspace*{0.1cm}

 Let $\varnothing\neq\mathcal{O}\subseteq\R^n$ be an arbitrary open set.
 The
 \emph{fractional Sobolev space} $H^s(\Oo)$ (of order $s\in(0,1)$) is the subset
 of $L^2(\Oo)$ defined
 as follows
$$H^s(\Oo) := \Big\{u\in L^2(\Oo):\,[u]^2_{s,\Oo} = \iint_{\Oo\times\Oo}
\frac{|u(x)-u(y)|^2}{|x-y|^{N+2s}}\,dx\,dy < +\infty\Big\}.$$
 We then list the few basic properties of $H^s(\Oo)$
we will exploit in this paper.
\begin{itemize}
 \item[a)] $H^s(\Oo)$ is a real Hilbert space, with the scalar product
 $$\langle u,v\rangle_{s,\,\Oo} 
 := \iint_{\Oo\times\Oo}\frac{(u(x)-u(y))(v(x)-v(y))}{|x-y|^{N+2s}}\,dx\,dy\qquad
 (u,v\in H^s(\Oo)).$$ 
 
 \item[b)] $C_0^\infty(\Oo)$ is a \emph{linear subspace of $H^s(\Oo)$}; in addition,
 in the particular case when $\Oo = \R^n$, we have that
  $C_0^\infty(\R^n)$ is \emph{dense} in $H^s(\R^n)$.
 \vspace*{0.1cm}
 
 \item[c)] If $\Oo = \R^n$ or if $\Oo$ has \emph{bounded boundary $\partial\Oo\in C^{0,1}$},
 we have the \emph{continuous embedding} $H^1(\Oo) \hookrightarrow H^s(\Oo)$,
 that is, there exists $\mathbf{c} = \mathbf{c}(n,s) > 0$ s.t.
 \begin{equation} \label{eq:H1embeddingHs}
  \iint_{\Oo\times\Oo}\frac{|u(x)-u(y)|^2}{|x-y|^{n+2s}}\,dx\,dy \leq 
  \mathbf{c}\,\|u\|_{H^1(\Oo)}^2\quad\text{for every $u\in H^1(\Oo)$}.
 \end{equation}
  In particular, if $\Oo\subseteq\R^n$ is a
  \emph{bounded open set} (with no regularity as\-sump\-tions on $\partial\Oo$) and if
  $u\in H_0^1(\Oo)$, setting $\hat{u} = u\cdot\mathbf{1}_\Oo\in H^1(\R^n)$ we have
  \begin{equation} \label{eq:H01embeddingHs}
  \iint_{\R^{2n}}\frac{|\hat{u}(x)-\hat{u}(y)|^2}{|x-y|^{n+2s}}\,dx\,dy \leq 
  \beta\,\int_\Oo|\nabla u|^2\,dx,
 \end{equation}
 where $\beta > 0$ is a suitable constant depending on $n,s$ and on $|\Omega|$. Here
 and throughout,  $|\cdot|$ denotes the $n$-dimensional Lebesgue measure.
\end{itemize}
\vspace*{0.1cm}

  \noindent\textbf{ii)\,\,The space $\mathcal{X}^{1,2}(\Omega)$.}
  Now we have briefly recalled some basic facts
  regarding fractional Sobolev spaces, we are in a position
  to introduce the adequate functional setting for the study of mixed local-nonlocal operators.
  \vspace*{0.1cm}
  
  Let then $\Omega\subseteq\R^n$ be a bounded open set Lipschitz boundary $\de\Omega$.
  We define the space
  $\mathcal{X}^{1,2}(\Omega)$ as the com\-ple\-tion
  of $C_0^\infty(\Omega)$ with respect to the
  \emph{global norm} 
  $$\rho(u) := \left(\||\nabla u|\|^2_{L^2(\R^n)}+[u]^2_{s,\R^n}\right)^{1/2},
  \qquad u\in C_0^\infty(\Omega).$$
  Due to its relevance in the sequel, we also introduce a distinguished notation
  for the \emph{cone of the non-negative functions} in $\mathcal{X}^{1,2}(\Omega)$: we set
  $$\mathcal{X}^{1,2}_+(\Omega) := \{u\in\mathcal{X}^{1,2}(\Omega):\,\text{$u\geq 0$ a.e.\,in $\Omega$}\}.$$
  Since this norm $\rho$ is induced by the scalar product
    $$\mathcal{B}_{\rho}(u,v) := \int_{\R^n}\nabla u\cdot\nabla v\,dx
    + \langle u,v\rangle_{s,\R^n}$$
  (where $\cdot$ denotes the usual scalar product in 
    $\R^n$), the space $\mathcal{X}^{1,2}(\Omega)$
    is a real \emph{Hilbert space}; most importantly, since $\Omega$ is bounded
    and $\de\Omega$ is Lipschitz, by combining the above
    \eqref{eq:H1embeddingHs} with the classical Poincar\'e
    inequality we infer that
    \begin{equation*}
     \vartheta^{-1}\|u\|_{H^1(\R^n)}\leq \rho(u)\leq \vartheta\|u\|_{H^1(\R^n)}\qquad
    \text{for every $u\in C_0^\infty(\Omega)$},
    \end{equation*}
    where $\vartheta > 1$ is a suitable constant depending on $n,s$ and on $|\Omega|$.
    Thus, $\rho(\cdot)$ and the full $H^1$-norm in $\R^n$
   are \emph{actually equivalent} on the space $C^\infty_0(\Omega)$, so that
   \begin{equation} \label{eq:defX12explicit}
   \begin{split}
    \mathcal{X}^{1,2}(\Omega) & = \overline{C_0^\infty(\Omega)}^{\,\,\|\cdot\|_{H^1(\R^n)}} \\
    & = \{u\in H^1(\R^n):\,\text{$u|_\Omega\in H_0^1(\Omega)$ and 
    $u\equiv 0$ a.e.\,in $\R^n\setminus\Omega$}\}.
    \end{split}
   \end{equation}
   We explicitly observe that, on account of \eqref{eq:defX12explicit}, 
   the functions in $\mathcal{X}^{1,2}(\Omega)$ naturally
  satisfy the nonlocal Dirichlet condition 
  prescribed in problem \eqref{eq:Problem}$_\lambda$, that is,
  \begin{equation*}
   \text{$u\equiv 0$ a.e.\,in $\R^n\setminus\Omega$ for every $u\in\mathcal{X}^{1,2}(\Omega)$}.
   \end{equation*}
   \begin{remark}[Properties of the space $\mathcal{X}^{1,2}(\Omega)$] \label{rem:spaceX12}
    For a future reference, we list in this remark some properties
    of the function space $\mathcal{X}^{1,2}(\Omega)$ which will be repeatedly
    exploited in the rest of the paper.
    \begin{enumerate}
     \item[1)] Since both $H^1(\R^n)$ and $H_0^1(\Omega)$ are \emph{closed} under the
     maximum/minimum o\-pe\-ra\-tion, it is readily seen that
     $$u_{\pm}\in \mathcal{X}^{1,2}(\Omega)\quad\text{for every $u\in\mathcal{X}^{1,2}(\Omega)$},$$
     \noindent where $u_{+}= \max\{u,0\}$ and $u_{-}=\max\{-u,0\}$.
     \vspace*{0.05cm}
     
     \item[2)] Since we are assuming that $\de\Omega$ is smooth, from
     \eqref{eq:defX12explicit} we see that a function $u\in H^1(\R^n)\cap C(\overline{\Omega})$
     belongs to the space $\mathcal{X}^{1,2}(\Omega)$ \emph{if and only if}
     $$\text{$u\equiv 0$ pointwise on $\R^n\setminus\Omega$}.$$
     \item[3)] On account of \eqref{eq:H01embeddingHs}, for every
     $u\in\mathcal{X}^{1,2}(\Omega)$ we have
     \begin{equation} \label{eq:X12HsRn}
    [u]_{s,\R^n}^2 = 
     \iint_{\R^{2n}}\frac{|u(x)-u(y)|^2}{|x-y|^{n+2s}}\,dx\,dy \leq\beta\int_\Omega|\nabla u|^2\,dx.
   \end{equation}
   As a consequence,
   the norm $\rho$ is \emph{glo\-bal\-ly equivalent} on $\mathcal{X}^{1,2}(\Omega)$ to the 
   $H_0^1$-no\-rm: in fact, by \eqref{eq:X12HsRn} there exists a constant $\Theta > 0$ such that
   \begin{equation} \label{eq:equivalencerhoH01}
    \||\nabla u|\|_{L^2(\Omega)}\leq \rho(u)\leq \Theta\||\nabla u|\|_{L^2(\Omega)}
    \quad\text{for every $u\in\mathcal{X}^{1,2}(\Omega)$}.
   \end{equation}
   
   \item[4)] By the (local) Sobolev inequality, for every $u\in\mathcal{X}^{1,2}(\Omega)$ 
   we have
   \begin{align*}
    S_n\|u\|_{L^{2^*}(\Omega)}^2 & = \|u\|_{L^{2^*}(\R^n)}^2
    \leq \int_{\R^n}|\nabla u|^2\,dx\leq \rho(u)^2.
   \end{align*}
   This, together with H\"older's inequality
   (recall that $\Omega$ is \emph{bounded}), proves the \emph{continuous embedding}
   $\text{$\mathcal{X}^{1,2}(\Omega)\hookrightarrow L^{q}(\Omega)$ for every $1\leq q\leq 2^*$}.$
   \vspace*{0.1cm}
   
   \item[5)] By combining \eqref{eq:equivalencerhoH01} with the \emph{compact embedding} of 
   $H_0^1(\Omega)\hookrightarrow L^q(\Omega)$ (holding true for every $1\leq q < 2^*$),
   we derive that also the embedding
   $$\text{$\mathcal{X}^{1,2}(\Omega)\hookrightarrow L^{q}(\Omega)$ is compact
   for every $1\leq q< 2^*$}.$$
   As a consequence, if $\{u_k\}_k$ is a bounded sequence in $\mathcal{X}^{1,2}(\Omega)$, it is possible
   to find a (unique) function $u\in\mathcal{X}^{1,2}(\Omega)$ such that (up to a sub-sequence)
   \begin{itemize}
    \item[a)] $u_n\to u$ weakly in $\mathcal{X}^{1,2}(\Omega)$;
    \item[b)] $u_n\to u$ \emph{strongly} in $L^q(\Omega)$ for every $1\leq q < 2^*$;
    \item[c)] $u_n\to u$ pointwise a.e.\,in $\Omega$.
   \end{itemize}
   
   \noindent Clearly, since both $u_n$ (for all $n\geq 1$) and $u$ \emph{identically vanish}
    out of $\Omega$, see
   \eqref{eq:defX12explicit}, we can replace
   $\Omega$ with $\R^n$ in the above assertions b)-c).
    \end{enumerate}
    We will exploit these properties without any further comment.
   \end{remark}
    We now observe that, since the leading operator 
    of \eqref{eq:Problem}$_\lambda$ is given by the $\e$-de\-pen\-dent operator
    $\LL = -\Delta+\e(-\Delta)^s$, it follows 
    that the bilinear form naturally associated with $\LL$ is the following
    $$\mathcal{B}_{\e}(u,v) = \int_{\R^n}\nabla u\cdot\nabla v\,dx
    + \e\,\langle u,v\rangle_{s,\R^n};$$
    in its turn, this form $\mathcal{B}_\e$ induces the $\e$-dependent quadratic form
    $$\rho_\e(u) = \||\nabla u|\|^2_{L^2(\R^n)}+\e\,[u]^2_s\qquad
     (\text{for $u\in\mathcal{X}^{1,2}(\Omega)$}).$$
    While in this perspective it should seem more \emph{natural} to use
     the
    norm $\rho_\e$ in place of $\rho$ on the space $\mathcal{X}^{1,2}(\Omega)$, 
    it is readily seen that these two norms
    are indeed \emph{equivalent on $\mathcal{X}^{1,2}(\Omega)$}
    (and equivalent to the $H_0^1$-norm), \emph{uniformly with respect to $\e$}:
    in fact, taking into account \eqref{eq:equivalencerhoH01} (and since $0<\e\leq 1$), we have
    \begin{equation} \label{eq:equivalenceuniforme}
     \||\nabla u|\|_{H^1_0(\Omega)}\leq\rho_\e(u)\leq \rho(u)\leq \Theta\||\nabla u|\|_{H^1_0(\Omega)}
     \end{equation}
   for some $\Theta > 0$ only depending on $n,s$. On account of \eqref{eq:equivalenceuniforme},
   we can indifferently use $\rho_\e(\cdot),\,\rho(\cdot)$ and the $H_0^1$-norm
   to define the topology of the space $\mathcal{X}^{1,2}(\Omega)$, and this choice
   does not produce any dependence on $\e\in(0,1]$.
   \medskip
   
    \noindent\textbf{Notation.} We conclude this second part of the section
    with a short list of notation, which will be used in the sequel;
    here, as usual, $\e\in(0,1]$ is a fixed parameter.
   \vspace*{0.1cm}
   
    1)\,\,Given any open set $\mathcal{O}\subseteq\R^n$ (not necessarily bounded), we set
    \begin{equation*}
    \begin{split}
     \mathrm{a)}&\,\,\mathcal{B}_{\e,\mathcal{O}}(u,v) =
    \int_{\mathcal{O}}\nabla u\cdot\nabla v\,dx
    + \e\, \langle u,v\rangle_{s,\R^n}\,\,(\text{for $u,v\in\mathcal{X}^{1,2}(\Omega)$}); \\
    \mathrm{b)}&\,\,\mathcal{Q}_{\e,\Oo}(u) = \mathcal{B}_{\rho_{\e},\Oo}(u,u)\,\,(\text{for $u
    \in\mathcal{X}^{1,2}(\Omega)$}).
    \end{split}
    \end{equation*}
    Since $\mathcal{X}^{1,2}(\Omega)\subseteq H^1(\R^n)$
    (see \eqref{eq:defX12explicit}), 
    the above forms $\mathcal{B}_{\e,\mathcal{O}}$ and $\mathcal{Q}_{\e,\Oo}$ are 
    well-de\-fined; moreover,
    again by taking into account \eqref{eq:defX12explicit} we have
	\begin{itemize}
	 \item $\mathcal{B}_{\e,\Omega}(u,v) = \mathcal{B}_{\e,\R^n}(u,v)\equiv \mathcal{B}_\e(u,v)$
	 for all $u,v\in\mathcal{X}^{1,2}(\Omega)$;
	 \item $\mathcal{Q}_{\e,\Omega}(u) = \mathcal{Q}_{\e,\R^n}(u) 
	 \equiv \rho_{\e}(u)$ for all $u\in\mathcal{X}^{1,2}(\Omega)$.
	\end{itemize}

    2)\,\,Given any \emph{bounded} open set $\mathcal{O}\subseteq\R^n$, we set
   $$\|u\|_{H_0^1(\Oo)} := \||\nabla u|\|_{L^2(\Oo)} = \Big(\int_\Oo|\nabla u|^2\,dx\Big)^{1/2}.$$
  \noindent\textbf{iii) Weak sub/su\-persolutions of problem \eqref{eq:Problem}$_\lambda$.} Thanks to
  all the prelimi\-naries reviewed so far, we are ready to provide
  the precise definition of \emph{weak sub/su\-persolutions} of problem  \eqref{eq:Problem}$_\lambda$.
  Actually, for a reason which will be clear later on, we consider
  the more general problem
  \begin{equation} \label{eq:pbgeneral}
    \begin{cases}
     \LL u = f(x,u) & \text{in $\Omega$} \\
     u = 0 & \text{in $\R^n\setminus\Omega$}
     \end{cases}
    \end{equation}
  where $f:\Omega\times\R\to\R$ is an arbitrary Carath\'eodory function satisfying
  the {growth condition} 
  \eqref{eq:growth_on_f}, that is, {\em there exists a constant $C > 0$ such that}
  \begin{equation*}
   |f(x,t)|\leq C(1+|t|^{2^*-1})\quad\text{\emph{for a.e.\,$x\in\Omega$ and every $t \in\R$}}.
   \end{equation*}
   Clearly, problem \eqref{eq:Problem}$_\lambda$ is of the form \eqref{eq:pbgeneral}, with the choice
   $f(x,t) = \lambda t^p+t^{2^*-1}$.
  \begin{definition} \label{def:weaksol}
   Let $f:\Omega\times\R\to\R$
   be a Carath\'eodory function sa\-ti\-sfying the growth condition
   \eqref{eq:growth_on_f}.
   We say that a function $u\in\mathcal{X}^{1,2}(\Omega)$ is 
   \begin{itemize}
    \item[a)] a \emph{weak subsolution}
   (resp.\,\emph{supersolution}) of problem \eqref{eq:pbgeneral}
   if we have 
    \begin{equation} \label{eq:weakformSol}
    \begin{gathered}
     \mathcal{B}_{\rho}(u,\varphi) \leq\,[\text{resp.}\,\geq]\,\, \int_{\Omega}f(x,u)\varphi\,dx  \\
    \text{for every test function $\varphi\in C_0^\infty(\Omega)$,\,$\varphi\geq 0$ in $\Omega$}.$$
    \end{gathered}
    \end{equation}
   \item[b)] a \emph{weak solution} of problem
   \eqref{eq:pbgeneral} if $u$ is both a weak subsolution and a weak
   supersolution of the same problem;
    \item[c)] a \emph{weak subsolution}
   (resp.\,\emph{supersolution, solution}) of problem 
   \begin{equation} \label{eq:generalpbSign}
    \begin{cases}
     \LL u = f(x,u) & \text{in $\Omega$} \\
     u > 0 & \text{in $\Omega$} \\
     u = 0 & \text{in $\R^n\setminus\Omega$}
     \end{cases}
    \end{equation}
   if $u$ is a weak subsolution
   (resp.\,supersolution, solution) of problem \eqref{eq:pbgeneral} in the sense
   already specified, further satisfying 
   $$\text{$u > 0$ a.e.\,in $\Omega$}.$$
\end{itemize}
  \end{definition}
  Now we have introduced
   Definition \ref{def:weaksol}, we prove a regularity result for solu\-tions of problem \eqref{eq:pbgeneral} which will be
   used in the sequel.
  \begin{theorem} \label{thm:regulAntCozzi}
    Let $f:\Omega\times\R\to\R$
   be a Carath\'eodory function sa\-ti\-sfying the growth condition
   \eqref{eq:growth_on_f}, and suppose that there exists a weak solution
   $u_0\in\mathcal{X}^{1,2}(\Omega)$ of problem \eqref{eq:pbgeneral}
   \emph{(}in the sense of Definition \ref{def:weaksol}\emph{)}.
   Then, we have
   \begin{equation} \label{eq:C1regulCozzi}
    \text{$u\in C^{1,\alpha}(\overline{\Omega})$ for some $\alpha\in (0,1)$}.
    \end{equation}
  \end{theorem}
  \begin{proof}
   First of all we observe that, since the function $f$ satisfies the growth condi\-tion
   \eqref{eq:growth_on_f}, we can apply \cite[Theorem 1.1]{SVWZ2},
   ensuring that $u_0\in L^\infty(\Omega)$; from this, again by exploiting
   condition \eqref{eq:growth_on_f}, we derive that 
   \begin{equation}  \label{eq:NemBounded}
    g(x) := f(x,u(x))\in L^\infty(\Omega).
   \end{equation}
   On account of \eqref{eq:NemBounded}, and since $\Omega$ has smooth boundary, we can then invoke the
   \emph{global regularity result} for the weak solutions of the 
   \emph{$\LL$\,-\,Dirichlet problem}
   $$\begin{cases}
   \LL u = g & \text{in $\Omega$} \\
   u = 0 & \text{in $\R^n\setminus\Omega$}
   \end{cases}$$
   proved in \cite[Theorem 1.1]{AC2},
   from which we derive \eqref{eq:C1regulCozzi}. This ends the proof. 
  \end{proof}
  
To make the paper self-contained, we finally state a strong maximum principle and Hopf Lemma which is a mere combination of results already appearing in the literature, see \cite{AC2, BDVV, BMV}.
  
  \begin{theorem} \label{thm:WMPSMPHopfTogether}
   Let $u\in\mathcal{X}^{1,2}(\Omega)\cap C^1(\overline{\Omega})$ be a weak
   supersolution of
   \begin{equation} \label{eq:pbZeroWMPSMP}
    \begin{cases}
   \LL u = 0 & \text{in $\Omega$}, \\
   u = 0 & \text{in $\R^n\setminus\Omega$}.
   \end{cases}
   \end{equation}
   We also assume that $u\not\equiv 0$ in $\Omega$. Then
   \begin{itemize}
    \item[i)] $u > 0$ pointwise in $\Omega$ and $u = 0$ pointwise in $\de\Omega$;
    \item[ii)] $\de_\nu < 0$ pointwise on $\de\Omega$ 
    \emph{(}where $\nu$ is the outer exterior normal at $\de\Omega$\emph{)}.
   \end{itemize}
  \end{theorem}
  \begin{proof}
  We first observe that, since $u\in\mathcal{X}^{1,2}(\Omega)\cap C^1(\overline{\Omega})$, by 
  Remark \ref{rem:spaceX12} we have 
  \begin{equation} \label{eq:ugeqzeroWMP}
   \text{$u\in H^1(\R^n)$ and $u\equiv 0$ pointwise in $\R^n\setminus\Omega$};
   \end{equation}
  moreover, since $u$ is a weak supersolution of \eqref{eq:pbZeroWMPSMP}, we also have
  \begin{equation} \label{eq:LLugeqzeroWMP}
  \mathcal{B}_\varepsilon(u,\varphi) \geq 0\quad 
  \forall\,\,\varphi\in C_0^\infty(\Omega),\,\text{$\varphi\geq 0$ in $\Omega$}.
 \end{equation}
  Gathering \eqref{eq:ugeqzeroWMP}-\eqref{eq:LLugeqzeroWMP}, 
 we are then entitled to apply the Weak Maximum Principle for $\LL$
 in \cite[Theorem 1.2]{BDVV}, ensuring that $u\geq 0$ pointwise in $\R^n$; as a consequence, 
 since $u\geq 0,\,u\not\equiv 0$ in $\Omega$ and since $u\equiv 0$ on 
 $\R^n\setminus\Omega$ (see \eqref{eq:ugeqzeroWMP}), by
  combining the Strong Maximum Principle in \cite[Theorem 3.1]{BMV} (applied here with $f\equiv 0$)
  and the Hopf lemma in \cite[Theorem 1.2]{AC2} (remind that $u\in C^1(\overline{\Omega})$), we conclude that
  $$\text{$u > 0$ pointwise in $\Omega$}\quad\text{and}\quad
   \text{$\de_{\nu}u < 0$ on $\de\Omega$}.$$
   This ends the proof.
  \end{proof}
    \begin{lemma}\label{lem:bound_indip_eps}
    Let $n >3$ and $h = h(x)\in L^{p}(\Omega)$ with $p>\tfrac{n}{2}$. 
    Assume that there exists a weak solution 
    $u\in \mathcal{X}^{1,2}(\Omega)$ of the Dirichlet problem
    \begin{equation*}
    \left\{ \begin{array}{rl}
    \LL u = h & \textrm{in } \Omega,\\
    u=0 & \textrm{in } \mathbb{R}^{n}\setminus \Omega.
    \end{array}\right.
    \end{equation*}
    Then $u\in L^{\infty}(\mathbb{R}^{n})$ and 
    \begin{equation*}
    \|u\|_{L^{\infty}(\mathbb{R}^{n})} \leq C\, \|h\|_{L^{p}(\Omega)},
    \end{equation*}
    \noindent for some positive constant $C>0$ independent of $\varepsilon$.
\begin{proof}
It is enough to closely follow the proof of \cite[Theorem 4.7]{BDVV} in the case 
$\varepsilon=1$. Take $\delta>0$ to be chosen later on and define the functions
\begin{equation*}
\tilde{u}:= \dfrac{\delta \,u}{K} \quad \textrm{and } \quad \tilde{v}:=\dfrac{\delta \, h}{K},
\end{equation*}
\noindent where $K = K(u,h,\Omega,p,n) := \|u\|_{L^{2^{\ast}}(\Omega)} + \|h\|_{L^{p}(\Omega)}$. One can now perform the same computations made in \cite[Theorem 4.7]{BDVV}, getting rid of the nonlocal part thanks to \cite[Equation (4.47)]{BDVV}, reaching the following:
\begin{equation}\label{eq:4.54_BDVV}
u(x) \leq \dfrac{K}{\delta} = \dfrac{\|u\|_{L^{2^{\ast}}(\Omega)} + \|h\|_{L^{p}(\Omega)}}{\delta}, \quad \textrm{for every } x\in \Omega,
\end{equation}
\noindent where $\delta>0$ has been chosen conveniently small but independently of $\varepsilon$.
Now, by Sobolev and H\"{o}lder inequalities, and testing the equation with $u$ itself, we get
\begin{equation*}
\mathcal{S}_{n} \, \|u\|^{2}_{L^{2^{\ast}}(\Omega)} \leq |\Omega|^{1/q} \, \|u\|_{L^{2^{\ast}}(\Omega)}\,\|h\|_{L^{p}(\Omega)},
\end{equation*}
\noindent where $\tfrac{1}{2^{\ast}}+\tfrac{1}{p}+\tfrac{1}{q}=1$ are the H\"{o}lder exponents. Therefore,
\begin{equation*}
\|u\|_{L^{2^{\ast}}(\Omega)} \leq \dfrac{|\Omega|^{1/q}}{\mathcal{S}_{n}} \, \|h\|_{L^{p}(\Omega)},
\end{equation*}
\noindent which combined with \eqref{eq:4.54_BDVV} gives the desired conclusion.
\end{proof}    
    \end{lemma}
    
 \noindent\textbf{iv)} {\bf First positive solution to \eqref{eq:Problem}$_{\lambda}$.} 
 We conclude this section of preliminaries by 
 spending a few words on the applicability of Theorem \ref{thm:ALTRI}
 (which is proved in the case when $\e = 1$)
 to our operator  $\LL$.
 \vspace*{0.1cm}
 
 In the case $\varepsilon =1$, the existence of a first positive solution is proved in \cite{AMT} by means of a classical sub and supersolution scheme that can be performed in the same way even replacing $\mathcal{L}_{1}$ with $\LL$. In particular, one can define
 \begin{equation}\label{eq:defLambda_epsilon}
 \Lambda_{\varepsilon}:= \sup \left\{\lambda >0: \eqref{eq:Problem}_{\lambda} \textrm{ admits a weak solution }\right\}.
 \end{equation}
 One then has to show what follows:
 \begin{itemize}
 \item[a)] $\Lambda_{\varepsilon}$ is well defined and $\Lambda_{\varepsilon}<+\infty$;
 \item[b)] problem \eqref{eq:Problem}$_{\lambda}$ admits a weak solution for every $0<\lambda\leq \Lambda_{\varepsilon}$.
 \end{itemize}
 For every {\it fixed} $0<\e\leq 1$,
 assertions a)\,-\,b) can be proved exactly as in \cite{AMT}. In particular, following the proof of \cite[Theorem 1.1]{AMT}, where the authors used a recent result in critical point theory proved in \cite{Moameni}, there exists $\mu_{\sharp,\e}>0$ (a priori depending on $\varepsilon$ in a not explicit way) such that a first solution 
 of \eqref{eq:Problem}$_\lambda$ (for $0<\lambda<\mu_{\sharp,\e}$) is obtained as a critical point over the set
\begin{equation*}
M = \left\{ v\in \mathcal{X}^{1,2}(\Omega): \|v\|_{L^{\infty}(\Omega)}\leq r, u\geq 0 \textrm{ a.e. in } \Omega \right\}.
 \end{equation*} 
 This is the reason why we need the following
 
 \begin{lemma} \label{lem:existencemusharp}
 There exists $\mu_{\sharp}>0$ independent of $\varepsilon$ such that \eqref{eq:Problem}$_{\lambda}$ possesses a minimal weak solution $\bar u_{\lambda,\e}\in \mathcal{X}^{1,2}(\Omega)\cap L^{\infty}(\Omega)$ for every $0<\lambda\leq\mu_{\sharp}$ and every $\e\in(0,1)$, which is increasing w.r.t.\,$\lambda$ and further satisfying that
 \begin{equation} \label{eq:minimalLinf}
  \|\bar u_{\lambda,\e}\|_{L^{\infty}(\Omega)}\ \leq C,
  \end{equation}
 \noindent where $C>0$ is independent of $\varepsilon$.
 \end{lemma}
\begin{proof}
As already discussed, Theorem \ref{thm:ALTRI}
applies to our problem \eqref{eq:Problem}$_{\lambda}$, the role of $\e$ being immaterial;
as a consequence, 
there exists $\Lambda_\e > 0$ such that problem
  \eqref{eq:Problem}$_{\lambda}$ possesses a \emph{minimal weak 
  solution}
  $$\bar u_{\lambda,\e}\in\mathcal{X}^{1,2}(\Omega)$$
  for every $0<\lambda\leq \Lambda_\e$; moreover, if $0<\lambda_1<\lambda_2<\Lambda_\e$, we have
  $$\text{$\bar u_{\lambda_1,\e}\leq \bar u_{\lambda_2,\e}$ a.e.\,in $\Omega$}.$$  
  We claim that there exists a constant $\lambda_{\sharp} > 0$,
  independent of $\e\in (0,1)$, such that
  \begin{equation} \label{eq:existencemusharp}
   \Lambda_\e\geq \lambda_\sharp.
  \end{equation}
  Indeed, let $\e\in (0,1)$ be fixed. By arguing as in the proof of
  \cite[Theorem 1.2]{AMT} we see that there exists
  $\Lambda_\e\geq \mu_{\sharp,\e}$, where $\mu_{\sharp,\e} > 0$ 
  satisfies the following property
  (see, precisely, the proof of \cite[Lemma 3.9]{AMT}):
  \begin{equation*}
   \begin{gathered}
   (\star)\qquad\quad\text{for every $0<\lambda<\lambda_{\sharp,\e}$ 
   there exist $0<r_1<r_2$ such that} \\
      \text{$C_{\e}(r^{{2^*}-2} + \lambda r^{p-1})\leq 1$ for all $r\in [r_1,r_2]$};   
   \end{gathered}
  \end{equation*}
  here, $C_\e > 0$ is a positive constant such that, for every $h\in L^p(\Omega)$ (with $p > n/2$)
  and every weak solution
  $u\in\mathcal{X}^{1,2}(\Omega)$ of the problem
  \begin{equation} \label{eq:pdDirCe}
   \begin{cases}
   \LL u = h & \text{in $\Omega$} \\
   u = 0 & \text{in $\R^n\setminus\Omega$}
   \end{cases}
  \end{equation}
  we have the \emph{a\,-\,priori estimate}
  \begin{equation} \label{eq:aprioriCe}
   \|u\|_{L^\infty(\R^n)}\leq C_\e\|h\|_{L^p(\Omega)}
  \end{equation}
  (the existence of such a constant follows from \cite[Theorem 4.7]{BDVV}). 
  Hence, the lower bound $\lambda_{\sharp,\e}$
  depends on $\e$ \emph{only through the constant $C_\e$ in estimate \eqref{eq:aprioriCe}}.
  On the other hand, by Lemma \ref{lem:bound_indip_eps},
  the constant $C_\e$ can be chosen \emph{indepen\-dently of $\e$}: this means, precisely,
  that there exists $C_0 > 0$ such that
  \begin{equation*}
   \|u\|_{L^\infty(\R^n)}\leq C_0\|h\|_{L^p(\Omega)}
  \end{equation*}
  for every $h\in L^p(\Omega)$ (with $p > n/2$), every weak solution $u\in\mathcal{X}^{1,2}(\Omega)$ of
  \eqref{eq:pdDirCe} \emph{and every $\e\in(0,1]$}. 
  Hence, if we choose $\mu_\sharp > 0$ in such a way that
  $(\star)$ holds with $C_0$ in place of $C_\e$
  (that is, if we choose the \emph{very same constant considered in \cite{AMT}}
  and corresponding to the case $\e = 1$), we conclude that
  \eqref{eq:existencemusharp} is satisfied. 
  \vspace*{0.05cm}
  
  We now turn to prove \eqref{eq:minimalLinf}. To this end we notice that,
  by proceeding exactly as in the proof of \cite[Theorem 1.1]{AMT} 
  (and by recalling that our constant $\mu_\sharp$ is precisely the one considered in \cite{AMT}
  and corresponding to the case $\e = 1$), there exists
  a solution $w_{\lambda,\e}\in\mathcal{X}^{1,2}(\Omega)$ of problem 
  \eqref{eq:Problem}$_\lambda$ satisfying the a-priori estimate
  $$\|w_{\lambda,\e}\|_{L^\infty(\Omega)}\leq r,$$
  where $r\in[r_1,r_2]$ is arbitrarily chosen, and $r_1<r_2$ are as in $(\star)$
  (and they depend on the fixed $\lambda$); 
  thus, by the minimality property
  of $\bar{u}_{\lambda,\e}$ we infer that
  $$0\leq \bar{u}_{\lambda,\e}\leq w_{\lambda,\e}\leq r\quad\text{a.e.\,in $\Omega$},$$
  and this readily gives the desired \eqref{eq:minimalLinf}.
 \end{proof} 
  \begin{remark} \label{rem:upperboundLambdae}
  Arguing exactly as in \cite[Proof of Theorem 1.2]{AMT} one can also realize that there exists $\Lambda_{\star}>0$ (indicated there as $\tilde{\lambda}$), depending on the first Dirichlet eigenvalue of $\mathcal{L}_{1}$ (see e.g. \cite{BDVV3}) but independent of $\varepsilon$, such that 
  $$\Lambda_{\varepsilon}\leq \Lambda_{\star}.$$
\end{remark}
\section{Purely sublinear problems and the proof of Theorem \ref{thm:main}} \label{sec:sublinear}
In this short section we investigate
 the \emph{purely sublinear} problem
\begin{equation}\label{eq:SublinearProblem}
\left\{ \begin{array}{rl}
\LL u  = \lambda \,u^{p}  & \textrm{in } \Omega,\\
u>0 & \textrm{in } \Omega,\\
u= 0 & \textrm{in } \mathbb{R}^{n}\setminus \Omega,
\end{array}\right.
\end{equation}
To be more precise, we establish/review some results which will be used (as funda\-mental tools)
in the proof of Theorem \ref{thm:main2}, and we demonstrate our first main result,
namely Theorem \ref{thm:main}.
\medskip

To begin with, we prove the following proposition.
\begin{proposition}\label{prop:SublinearGeneral}
 Let $\e\in(0,1]$ be fixed. Moreover, let $\lambda > 0$ and $p\in(0,1)$.
 
 Then, the following assertions hold.
 \begin{itemize}
  \item[i)] There exists a \emph{unique weak solution} $w_{\lambda,\e} 
  \in \mathcal{X}^{1,2}(\Omega)$
  of problem \eqref{eq:SublinearProblem}, which is the unique \emph{global minimizer of the functional}
  \begin{equation} \label{eq:funJlambdaSub}
   J_{\lambda,\e}(u) = 
  \frac{1}{2}\rho_\e(u)^2-\frac{\lambda}{p+1}\int_\Omega|u|^{p+1}\,dx.
  \end{equation}
  Moreover, $J_{\lambda,\e}(w_{\lambda,\e}) < 0$.
  \medskip
  
  \item[ii)]  $w_{\lambda,\e}\in C^{1,\alpha}(\overline{\Omega})$ for every $\alpha\in (0,1)$, and
  $\de_\nu w_{\lambda,\e} < 0$ on $\de\Omega$.
  \medskip
  
  \item[iii)] There exists $\e_0 > 0$
  with the following property:
  \emph{given any ball $B_{R}(x_0)\subseteq\Omega$ and any $0<r\leq \min\{1,R\}$, for every $0<\e\leq\e_0$
  we have}
  $$w_{\lambda,\e}\geq c_2\quad\text{a.e.\,on $B_r(x_0)$},$$
  for some constant $c_2 > 0$ independent of $\e$.
 \end{itemize}
 \end{proposition}
 We explicitly stress that, since problem \eqref{eq:SublinearProblem} is of the form
 \eqref{eq:pbgeneral} {(}with $f = \lambda t^p${)}, the definition of weak solution of 
 \eqref{eq:SublinearProblem}
 is given in Definition \ref{def:weaksol}.
 \begin{proof}
 We prove separately the three assertions.
 \medskip
 
 \noindent i)\,\,To begin with, by using \cite[Theorem 1.2]{BMV} with the choice
 $f(t) = \lambda t^p$ (notice that this theorem certainly
 applies to the operator $\LL$, the role of $\e$ being immaterial)
 we immediately get the existence of a \emph{unique weak solution} 
 $$w_{\lambda,\e}\in\mathcal{X}^{1,2}(\Omega)\cap L^\infty(\Omega)$$
 of problem \eqref{eq:SublinearProblem}; moreover, owing to
 \cite[Proposition 6.2]{BMV} (which also
 can be applied to the $\e$\,-\,dependent operator $\LL$), we know that
 this $w_{\lambda,\e}$ satisfies
 \begin{equation} \label{eq:globalminwlambdaJ}
  J_{\lambda,\e}(w_{\lambda,\e}) = 
  \min_{\varphi\in\mathcal{X}^{1,2}(\Omega)}J_{\lambda,\e}(\varphi)\quad\text{and}\quad
 J_{\lambda,\e}(w_{\lambda,\e}) < 0.
\end{equation}
 
 \noindent ii)\,\,Since $f(t) = \lambda t^p$ clearly satisfies 
 the growth condition \eqref{eq:growth_on_f} (with $C_f = 1$),
 the regularity up the boundary of $w_{\lambda,\e}$ follows
 from Theorem \ref{thm:regulAntCozzi}; this, jointly with the fact that $w_{\lambda,\e} > 0$
 in $\Omega$, allows us to apply \ref{thm:WMPSMPHopfTogether}, obtaining
 $$\text{$\de_{\nu}w_{\lambda,\e} < 0$ pointwise on $\de\Omega$}.$$
 
 \noindent iii)\,\,The proof of this assertion
 is split into three steps, and it
 is very similar to that of
 \cite[Proposition 3.7]{BV2}; we present it in detail
 for the sake of completeness. 
 \medskip 
 
 \textsc{Step 1).} In this first step, we prove that 
 \begin{equation} \label{eq:wlambdaeunfH01}
  \|w_{\lambda,\e}\|_{H^1_0(\Omega)}\leq \lambda^{\frac{1}{1-p}}\,c\quad\text{for all $\e\in(0,1]$}.
 \end{equation}
 for some constant $c > 0$ only depending on $\Omega$ and on $p$.
 
  Indeed, since $w_{\lambda,\e}\in\mathcal{X}^{1,2}(\Omega)$ 
 is a weak solution
 of problem \eqref{eq:SublinearProblem}  (in the sense of
 De\-fi\-nition \ref{def:weaksol}), by choosing $w_{\lambda,\e}$ as a test
 function in \eqref{eq:weakformSol}, we get
  \begin{align*}
    \|w_{\lambda,\e}\|_{H^1_0(\Omega)}^2 &
    \leq \rho_\e(w_{\lambda,\e})^2
    = \frac{\lambda}{p+1	}\int_{\Omega}w_{\lambda,\e}^{p+1}\,dx \\
    & (\text{using H\"older's and Poincar\'e inequality}) \\   
    & \leq \frac{\lambda|\Omega|^{1-\frac{p+1}{2}}}{p+1}\|w_{\lambda,\e}\|^{p+1}
    \leq \lambda\,c(\Omega,p)\|w_{\lambda,\e}\|^{p+1}_{H^1_0(\Omega)},
  \end{align*}
  where $c(\Omega,p) > 0$ is a suitable constant only depending on $\Omega$ and on $p$.
  From this, since $w_{\lambda,\e}\not\equiv 0$ (and since $0<p<1$), we immediately derive the
  claimed \eqref{eq:wlambdaeunfH01}.
  \medskip
  
  \textsc{Step 2).} In this second step we prove that, as $\e\to 0^+$, we have
  \begin{itemize}
   \item[a)] $w_{\lambda,\e}\to w_{\lambda}$ weakly in $\mathcal{X}^{1,2}(\Omega)$;
   \item[b)] $w_{\lambda,\e}\to w_{\lambda}$ strongly in $L^m(\Omega)$ for every $1\leq m < 2^*$;
  \end{itemize}
  where $w_{\lambda}\in H^1_0(\Omega)$ is the unique solution of the purely local problem
  \begin{equation} \label{eq:singularsololocale}
   \begin{cases}
 -\Delta u = \lambda u^{p} & \textrm{in $\Omega$},\\
 u>0 & \textrm{in $\Omega$},\\
 u= 0 & \textrm{in $\mathbb{R}^n \setminus \Omega$},
 \end{cases}
 \end{equation}
 To this end, we let $\{\e_j\}_j\subseteq(0,1]$ be any sequence converging to $0$ as $j\to+\infty$, and 
 we arbitrarily choose a subsequence
 $\{\e_{j_k}\}_k$  of $\{\e_j\}_j$. On account of
 \eqref{eq:wlambdaeunfH01}, and using
 Remark \ref{rem:spaceX12}\,-\,5), we know that there exists
 some function $\bar{w}\in\mathcal{X}^{1,2}(\Omega)$ such that, as $k\to+\infty$ and
 up to possibly 
 choosing a further subsequence,
 \begin{itemize}
   \item $w_{\lambda,\e_{j_k}}\to \bar{w}$ weakly in $\mathcal{X}^{1,2}(\Omega)$;
   \item $w_{\lambda,\e_{j_k}}\to \bar{w}$ strongly in $L^m(\Omega)$ for every $1\leq m < 2^*$;
  \end{itemize}
  from this, since $w_{\lambda,\e}$ is a global minimizer for $J_{\lambda,\e}$
  (see \eqref{eq:globalminwlambdaJ}), we get
  \begin{align*}
   J_{\lambda}(\bar{w})
   & = \frac{1}{2}\|\bar{w}\|^2_{H^1_0(\Omega)}
   -\frac{\lambda}{p+1}\int_\Omega |\bar{w}|^{p+1}\,dx
   \\
   & \leq \liminf_{k\to+\infty}
   J_{\lambda,\e_{j_k}}(w_{\lambda,\e_{j_k}}) 
    \\
   & \leq \liminf_{k\to+\infty}
   J_{\lambda,\e_{j_k}}(\varphi) = J_{\lambda}(\varphi)\quad\text{for every $\varphi\in H_0^1(\Omega)$},
  \end{align*}
  and thus $\bar{w}$ is a global minimizer for the functional $J_{\lambda}$
  naturally associated with the purely local, singular problem 
  \eqref{eq:singularsololocale}. As a consequence, since $J_{\lambda}$ as a
  unique global minimizer which is the unique solution $w_\lambda$ of \eqref{eq:singularsololocale}
  (that is, assertion i) also holds in the purely local case, see
  \cite{BO}),
  we derive that
  $$\bar{w} = w_{\lambda}.$$
  Due to arbitrariness of the sequence 
  $\{\e_j\}_j\subseteq(0,1]$ of its subsequence
 $\{\e_{j_k}\}_k$, and since the limit
 \emph{is always the same}, we conclude the validity of a)\,-\,b).
 \medskip
 
 \textsc{Step 3)}\,\,In this last step, we complete the proof of
 the assertion. First of all we observe that, since
 $w_{\lambda,\e}$ is a weak solution of problem \eqref{eq:SublinearProblem},
 we readily derive that $w_{\lambda,\e}$ is a also \emph{weak supersolution}
 (in the sense of Definition \ref{def:weaksol}) of
 $\LL u = 0$ in $\Omega$;
 thus, by \cite[Proposition 3.3]{BV2} (and since $w_{\lambda,\e}\geq 0$ a.e.\,in $\R^n$) we get
 \begin{equation} \label{eq:HarnackwlambdaeI}
  w_{\lambda,\e}\geq c\,\left(\fint_{B_r(x_0)}w_{\lambda,\e}^Q \, dx\right)^{1/Q}\quad\text{a.e.\,on $B_r(x_0)$},
 \end{equation}
 where $Q,\,c > 0$ do not depend on $\e$.
 On the other hand, since by \textsc{Step 2)} we know that $w_{\lambda,\e}\to w_\lambda$
 as $\e\to 0$ strongly in $L^m(\Omega)$ for every $1\leq m <2^*$, we have
 \begin{equation} \label{eq:convergenceMeanwlambdae}
  \lim_{\e\to 0^+}\fint_{B_r(x_0)}w_{\lambda,\e}^Q \, dx = \fint_{B_r(x_0)}w_{\lambda}^Q \, dx.
 \end{equation}
 Gathering \eqref{eq:HarnackwlambdaeI}-\eqref{eq:convergenceMeanwlambdae}, 
 and recalling that $w_\lambda$ is the unique solution of \eqref{eq:singularsololocale} (hence,
 $w_\lambda > 0$ a.e.\,in $\Omega$),
  we can then fined some $\e_0 > 0$ such that
  $$w_{\lambda,\e}\geq \frac{c}{2}\left(\fint_{B_r(x_0)}w_{\lambda}^Q \, dx\right)^{1/Q} > 0$$
  a.e.\,in $B_r(x_0)$ and \emph{for every $0<\e\leq \e_0$}. This ends the proof.
 \end{proof}
 We now turn to prove our first main result, namely Theorem \ref{thm:main2}.
 Before doing this, we establish the following auxiliary lemma.
 \begin{lemma}\label{lem:Lemma3.6_ABC}
 Let $0<\e\leq 1,\,\lambda > 0$ be fixed, and let $w_{\lambda,\e} \in \mathcal{X}^{1,2}(\Omega)$ 
 be the unique solution of problem \eqref{eq:SublinearProblem} 
 \emph{(}whose existence and uniqueness is guaranteed by
 Proposi\-tion \ref{prop:SublinearGeneral}\emph{)}.
 Then there exists a positive constant $\beta = \beta_{\lambda,\varepsilon} >0$ such that
 \begin{equation}\label{eq:Lemma3.6_ABC}
 \rho_{\varepsilon}(\varphi)^2 - \lambda p \int_{\Omega}w_{\lambda,\varepsilon}^{p-1}\varphi \, dx \geq 
 \beta_{\lambda,\e}\,\int_{\Omega}\varphi^2 \, dx, \quad \textrm{for all } \varphi \in \mathcal{X}^{1,2}(\Omega).
 \end{equation}
 \end{lemma}
 \begin{proof}
  We first remind that, on account of
  Proposition \ref{prop:SublinearGeneral}\,-\,i), the function $w_{\lambda,\e}$
  is the u\-ni\-que, global minimizer of the functional $J_{\lambda,\e}$ in \eqref{eq:funJlambdaSub}, that is,
  $$J_{\lambda,\e}(w_{\lambda,\e}) = \min_{u\in\mathcal{X}^{1,2}(\Omega)}J_{\lambda,\e}(u);$$
  this, together with the fact that $w_{\lambda,\e}\in C^{1}(\overline{\Omega})$, see Proposition \ref{prop:SublinearGeneral}\,-\,ii),
  ensures that the \emph{second variation} of the functional
  $J_{\lambda,\e}$ at $w_{\lambda,\e}$ is non-negative, namely
  \begin{equation}\label{eq:SecondVariationJat_v}
\rho_{\varepsilon}(\varphi)^2 - p\lambda\int_{\Omega}w_{\lambda,\varepsilon}^{p-1}\varphi^2\, dx \geq 0
\quad\text{for every $\varphi\in\mathcal{X}^{1,2}(\Omega)$}.
 \end{equation}
 As a consequence, we get
 \begin{equation} \label{eq:lambda1geqzero}
 \begin{split}
 & \lambda_1 (\mathcal{L}_{\varepsilon}-\lambda\,p\,w_{\lambda,\varepsilon}^{p-1})
 \\
 & \qquad
 = \inf\Big\{\rho_{\varepsilon}(\varphi)^2 - \lambda\,p\,
 \int_{\Omega}w_{\lambda,\varepsilon}^{p-1}\varphi^2\, dx:\,
 \varphi\in\mathcal{X}^{1,2}(\Omega),\,\|\varphi\|_{L^2(\Omega)} = 1\Big\}
 \geq 0.
 \end{split}
 \end{equation}
 With \eqref{eq:lambda1geqzero} 
 at hand, to prove \eqref{eq:Lemma3.6_ABC} it then suffices to show that
 \begin{equation} \label{eq:toprovelambda1positive}
 \lambda_1 (\mathcal{L}_{\varepsilon}-\lambda\,p w_{\lambda,\varepsilon}^{p-1}) > 0;
 \end{equation}
 in fact, if \eqref{eq:toprovelambda1positive} holds, the desired \eqref{eq:Lemma3.6_ABC} 
 follows with the choice 
 $$\beta_{\lambda,\e} = \lambda_1 (\mathcal{L}_{\varepsilon}-\lambda\,p w_{\lambda,\varepsilon}^{p-1}).$$
 Hence, we turn to prove \eqref{eq:toprovelambda1positive}.
  To this end, we argue by contradiction assuming that 
 \eqref{eq:toprovelambda1positive} is not true; thus,
 by \eqref{eq:lambda1geqzero} we necessarily have that
 $$\lambda_1 (\mathcal{L}_{\varepsilon}-\lambda\,p w_{\lambda,\varepsilon}^{p-1}) = 0.$$
 Taking into account Proposition \ref{prop:SublinearGeneral}\,-\,ii), and reminding that
 $w_{\lambda,\e} > 0$ \emph{po\-int\-wi\-se in $\Omega$} (as $w_{\lambda,\e}\in C^1(\overline{\Omega})$),
 we can find some $\delta = \delta_{\lambda,\e} > 0$ such that
 $$w_{\lambda,\e}(x)\geq \delta_{\lambda,\e}\,\mathrm{dist}(x,\de\Omega)\quad\text{for
 every $x\in\overline{\Omega}$};$$
 from this, by arguing exactly as in \cite[Remark 3.1]{ABC}, we infer the existence 
  of a \emph{non-negative eigenfunction}
  associate with $\lambda_1 (\mathcal{L}_{\varepsilon}-a(x))$; this means,
  precisely, that there exists a function
  $\psi_{\lambda,\e} \in \mathcal{X}^{1,2}(\Omega),\,\psi_{\lambda,\e}\gneqq 0$, 
  such that
 \begin{equation*}
\mathcal{B}_\e(\psi_{\lambda,\e},\varphi) -
\lambda\,p\int_\Omega w_{\lambda,\varepsilon}^{p-1}\psi_{\lambda,\e}
\,dx = 0\quad\text{for every $\varphi\in\mathcal{X}^{1,2}(\Omega)$}.
 \end{equation*}
 Now, choosing $\varphi = w_{\lambda,\varepsilon}$ in the above identity, we get
 \begin{equation}\label{eq:Testando_con_v}
 \mathcal{B}_{\varepsilon}(\psi_{\lambda,\varepsilon}, w_{\lambda,\varepsilon}) = 
 \lambda\,p\int_{\Omega}w_{\lambda,\varepsilon}^{p}\psi_{\lambda,\varepsilon}\, dx;
 \end{equation}
 on the other hand, if we choose $\varphi = \psi_{\lambda,\e}$
 as a test function for the equation solved by $w_{\lambda,\varepsilon}$
 appearing in problem \eqref{eq:SublinearProblem}
 (see Definition \ref{def:weaksol}), we have
 \begin{equation}\label{eq:Testando_con_varphi}
 \mathcal{B}_{\varepsilon}(\psi_{\lambda,\varepsilon}, w_{\lambda,\varepsilon}) 
  = \lambda\int_{\Omega}w_{\lambda,\varepsilon}^{p}\psi_{\lambda,\varepsilon}\, dx.
 \end{equation}
 Gathering \eqref{eq:Testando_con_varphi}\,-\,\eqref{eq:Testando_con_v}, we then conclude that 
 $$\lambda\int_{\Omega}w_{\lambda,\varepsilon}^{p}\psi_{\lambda,\varepsilon}\, dx = 
 \lambda\,p\int_{\Omega}w_{\lambda,\varepsilon}^{p}\psi_{\lambda,\varepsilon}\, dx$$
 which is a contradiction, since $p\in (0,1)$
 (also recall that $w_{\lambda,\e} > 0$ pointwise in $\Omega$, and that $\psi_{\lambda,\e}\gneqq 0$). 
 This ends the proof.
 \end{proof}
  With Lemma \ref{lem:Lemma3.6_ABC} at hand, we can provide the
\begin{proof}[Proof of Theorem \ref{thm:main}]
Let $\e\in(0,1]$ be fixed, and let
 $\beta_\e = \beta_{1,\varepsilon}>0$ be as in Lemma \ref{lem:Lemma3.6_ABC}, with
 $\lambda = 1$. Moreover, let $M_\e>0$ be such that 
\begin{equation}\label{eq:Def_M}
p \, M_\e^{2^{\ast}-1} < \beta_\e.
\end{equation}
We now prove that, \emph{with this choice of $M_\e$}, assertion 1) is satisfied.

To this end, we argue
by contradiction assuming that, for some fixed $\lambda\in (0,\Lambda_\e)$, there exists 
a second solution $v_{\lambda,\e}$ of problem \eqref{eq:Problem}$_{\lambda}$ such that 
\begin{equation}\label{eq:Stima_w_lambda}
\|v_{\lambda,\e}\|_{L^{\infty}(\mathbb{R}^{n})} \leq M_\e.
\end{equation}
By Theorem \ref{thm:ALTRI} (which can be applied
to $\LL$, see Lemma \ref{lem:existencemusharp}), 
we know that there exists a \emph{minimal solution $u_{\lambda,\e}$} of
problem
\eqref{eq:Problem}$_{\lambda}$; hence, we have
\begin{equation*}
v_{\lambda,\e} = u_{\lambda,\e} + g_\e, 
\end{equation*}
 for some function $g_\e\in\mathcal{X}^{1,2}(\Omega),\,g_\e\geq 0$ in $\Omega$. 
 
 We now define the auxiliary function $\zeta_\e:\mathbb{R}^{n}\to \mathbb{R}$ as
\begin{equation*}
\zeta_\e(x) := \lambda^{1/(1-p)}\,w_{1,\e}(x),
\end{equation*}
\noindent where $w_{1,\e}$ is the unique solution of the purely sublinear 
problem \eqref{eq:SublinearProblem} with $\lambda = 1$. 
A direct computation shows that $\zeta$ weakly solves
\begin{equation}
\mathcal{L}_\e\zeta_\e = \lambda \, \zeta_\e^{p} \quad \textrm{in } \Omega.
\end{equation}
On the other hand, 
since $u_{\lambda,\e}$ is a weak solution of
     problem \eqref{eq:Problem}$_\lambda$ (hence, in particular,
     $u_{\lambda,\e} > 0$ a.e.\,in $\Omega$), this function is a 
     \emph{weak supersolution} of  problem \eqref{eq:SublinearProblem} (in
     the sense of Definition \ref{def:weaksol} with $f(t) = \lambda t^p$);
     as a consequence, since we have already observed that
     $\zeta$ solves
     \eqref{eq:SublinearProblem}, we can use \cite[Lemma 3.11]{AMT}
     (which can certainly be applied
     to the operator $\LL$, the role of $\e$ being immaterial), obtaining
\begin{equation} \label{eq:Comparison_ulambda_v}
\text{$\zeta_\e = \lambda^{1/(1-p)}w_{1,\e} \leq u_{\lambda,\e}$ in $\R^n$}.
\end{equation}
With all this at hand we can now conclude. First,
\begin{equation*}
\begin{aligned}
\mathcal{L}_\e v_{\lambda,\e} &= \mathcal{L}_\e(u_{\lambda,\e} +g_\e)\\
&  = \lambda (u_{\lambda,\e}+g_\e)^{p} + (u_{\lambda,\e}+g_\e)^{2^{\ast}-1} \quad \mbox{($v_{\lambda,\e}$ is a solution of \eqref{eq:Problem}$_{\lambda}$)}\\
&\leq \lambda u_{\lambda,\e}^p + \lambda p\, u_{\lambda,\e}^{p-1} g_\e  + (u_{\lambda,\e}+g_\e)^{2^{\ast}-1} \quad \mbox{(by concavity of $t\mapsto t^{p}$)};
\end{aligned}
\end{equation*}
as a consequence, we have
\begin{equation*}
\begin{aligned}
\LL g_\e &\leq \lambda u_{\lambda,\e}^p + \lambda p\, u_{\lambda,\e}^{p-1} g_\e  
 + (u_{\lambda,\e}+g_\e)^{2^{\ast}-1} - \LL u_{\lambda,\e}\\
&= \lambda u_{\lambda,\e}^p + \lambda p\, u_{\lambda,\e}^{p-1} g_\e  + (u_{\lambda,\e}+g_\e)^{2^{\ast}-1} - \lambda u_{\lambda,\e}^{p} - u_{\lambda,\e}^{2^{\ast}-1} \\
&= \lambda p\, u_{\lambda,\e}^{p-1} g_\e  + (u_{\lambda,\e}+g_\e)^{2^{\ast}-1} - 
 u_{\lambda,\e}^{2^{\ast}-1} \\
&\leq p\, w_{1,\e}^{p-1} g_\e  + (u_{\lambda,\e}+g_\e)^{2^{\ast}-1} -  u_{\lambda,\e}^{2^{\ast}-1} \quad \mbox{(by \eqref{eq:Comparison_ulambda_v}, and since $0<p<1$)}\\
&\leq p\,w_{1,\e}^{p-1} g_\e  + p \, M_\e^{2^{\ast}-1} g_\e \qquad \mbox{(by \eqref{eq:Stima_w_lambda})}\\
&<p\,w_{1,\e}^{p-1} g_\e + \beta_\e \, g_\e \qquad \qquad \mbox{(by \eqref{eq:Def_M})}.
\end{aligned}
\end{equation*}
\noindent Summing up, the function $g_\e$ satisfies (in the weak sense on $\Omega$)
\begin{equation*}
\mathcal{L}_\e g_\e - p\,w_{1,\e}^{p-1}g_\e < \beta_\e g_\e.
\end{equation*}
\noindent Testing it with $g$ itself, and using \eqref{eq:Lemma3.6_ABC}
(with $\lambda = 1$), we get
\begin{equation*}
\beta_\e \int_{\Omega}g_\e^2 \, dx \leq \rho_\e(g_\e)^2 - 
 p \int_{\Omega}w_{1,\e}^{p-1}g_\e\, dx < \beta_\e \int_{\Omega}g_\e^2 \, dx,
\end{equation*}
\noindent which is a contradiction. 
\end{proof}
    
\section{Proof of Theorem \ref{thm:H1vsC1}}\label{sec.H1C1}
In this section we show that the famous result \cite[Theorem 1]{BNH1C1} still holds in the mixed local-nonlocal setting. 
The proof is a quite simple adaptation of the original proof by Brezis and Nirenberg and the major role is played by the recent regularity results established in \cite{AC2} and \cite{SVWZ2}.

\begin{proof}[Proof of Theorem \ref{thm:H1vsC1}]
By assumption, $u_0 \in \mathcal{X}^{1,2}(\Omega)$ is a local minimizer for $\Phi$ in the $C^1$-topology; this
means, precisely, that there exists some $r > 0$ such that
\begin{equation}\label{eq:C1Minimizer}
\Phi(u_0) \leq \Phi(u_0 +v), \quad \textrm{for all } v\in C^{1}_{0}(\overline{\Omega}) \, \textrm{ s.t. } \|v\|_{C^1(\Omega)}\leq r.
\end{equation}
It follows that the Euler-Lagrange equation satisfied by $u_0$ is 
\begin{equation*}
\begin{cases}
\mathcal{L}_{\varepsilon}v = f(x,v)& \textrm{ in } \Omega,\\
v=0 & \textrm{ in } \mathbb{R}^{n}\setminus \Omega;
\end{cases}
\end{equation*}
 as a consequence, since (by assumption) $f$ satisfies the growth condition \eqref{eq:growth_on_f}, from
 Theorem \ref{thm:regulAntCozzi} we derive that $u_0\in C^{1,\alpha}(\overline{\Omega})$
 for some $\alpha\in(0,1)$. In view of this fact,
 from now on we may assume without loss of generality that
 $$u_0 = 0.$$
  We now argue by contradiction, assuming that $u_0$ 
  \emph{is not a local minimizer for $\Phi$ in the $\mathcal{X}^{1,2}$-to\-po\-lo\-gy}; thus,
  for every $\delta > 0$ there exists $v_\delta\in\mathcal{X}^{1,2}(\Omega)$ such that
 \begin{equation} \label{eq:tocontradictH1C1PartI}
  \rho_\e(v_\delta)<\delta\quad\text{and}\quad \Phi(v_\delta) < \Phi(0).
 \end{equation}
 We explicitly notice that use of $\rho_\e(\cdot)$ to define an open neighborhood
 of $u_0 = 0$ in \eqref{eq:tocontradictH1C1PartI} is motivated by the fact that, since $\e\in (0,1)$,
 by \eqref{eq:equivalenceuniforme} we have
 $$\|u\|_{H^1_0(\Omega)}\leq\rho_\e(u)\leq \rho(u)\leq \Theta\|u\|_{H^1_0(\Omega)}$$
 for some $\Theta > 0$ only depending on $n,s$; as a consequence,
 the norm $\rho_\e(\cdot)$ is globally equivalent to $\rho(\cdot)$ (and to the 
 $H_0^1$-norm), uniformly w.r.t.\,$\e$.
 \vspace*{0.05cm}
 
 Now, since the closed set $K_\delta = \{u\in\mathcal{X}^{1,2}(\Omega):\,\rho_\e(u)\leq \delta\}
 \subseteq$ 
 is weakly sequentially compact, and since $\Phi$
 is weakly lower semincontinuous and bounded from below
 (by the growth assumption \eqref{eq:growth_on_f} on $f$), 
 we can find $w_\delta\in K_\delta$ such that
 $$\Phi(w_\delta) = \min_{u\in K_\delta}\Phi(u).$$
 In particular, by \eqref{eq:tocontradictH1C1PartI} we get
 \begin{equation} \label{eq:tocontradictH1C1}
  \Phi(w_\delta) \leq \Phi(v_\delta) < \Phi(0)\quad\text{for every $\delta > 0$}.
 \end{equation}
 We then claim that 
 \begin{equation} \label{eq:wdeltatozeroCLAIM}
  \text{$\mathrm{i)}\,\,w_\delta\in C_0^1(\overline{\Omega})$ \quad and\quad
    $\mathrm{ii)}\,\,\lim_{\delta \to 0}w_\delta = 0$ in the $C^1(\overline{\Omega})$-topology}.
 \end{equation}
 Taking this claim for granted for a moment, we can immediately
 complete the proof of the theorem: in fact, owing to 
 \eqref{eq:wdeltatozeroCLAIM}, we see that there exists $\delta_0 > 0$ such that
 $$\|w_\delta\|_{C^1(\overline{\Omega})}\leq r\quad\text{for every $0<\delta<\delta_0$},$$
 where $r > 0$ is as in \eqref{eq:C1Minimizer}; thus, by \eqref{eq:C1Minimizer}
 (and since we are assuming $u_0 = 0$) we infer that
 $\Phi(0)\leq \Phi(w_\delta)$ for every $0<\delta<\delta_0$,
 but this clearly contradicts \eqref{eq:tocontradictH1C1}.
 
 We then turn to prove \eqref{eq:wdeltatozeroCLAIM}, and we proceed by steps.
 \medskip
 
 \textsc{Step I).} In this first step we prove that there exist $\delta_0 > 0$
 and a constant $c_0 > 0$, \emph{independent of $\delta$}, such that
 the following estimate holds:
 \begin{equation} \label{eq:wdeltaInfStepI}
  \|w_\delta\|_{L^\infty(\Omega)}\leq c_0\quad\text{for every $0<\delta<\delta_0$}.
 \end{equation}
 To this end we first notice that, since $w_\delta$ is a global minimizer for
 $\Phi$ on $K_\delta$, 
 by the Lagrange Multiplier Theorem, there exists $\mu_{\delta}\in \mathbb{R}$ such that 
\begin{equation}\label{eq:wdeltaminI}
\mathcal{B}_{\varepsilon}(w_{\delta},\varphi) - \int_{\Omega}f(x,w_{\delta})\varphi \, dx = \mu_{\delta}\, \mathcal{B}_{\varepsilon}(w_{\delta},\varphi) \quad \textrm{for all } \varphi \in \mathcal{X}^{1,2}(\Omega).
\end{equation}
If $\rho_{\varepsilon}(v_{\delta}) < \delta$ (that is, if $w_\delta\in\mathrm{int}(K_\delta)$), 
then $\mu_{\delta}=0$; if, instead,
$\rho_{\varepsilon}(v_{\delta}) = \delta$ (that is, if $w_\delta\in\de K_\delta$), 
by choosing $\varphi = v_{\delta}$ in \eqref{eq:wdeltaminI} we derive that
\begin{equation*}
\begin{aligned}
\mu_{\delta} \, \delta^2 & = \mathcal{B}_\e(w_\delta,w_\delta) = 
\mathcal{B}_{\varepsilon}(w_{\delta},w_{\delta}) - \int_{\Omega}f(x,w_{\delta})w_{\delta}\, dx \\
& = \Phi(w_{\delta}) + \int_{\Omega}F(x,w_{\delta})\, dx - \int_{\Omega}f(x,w_{\delta})w_{\delta}\, dx\\
& (\text{using \eqref{eq:tocontradictH1C1}, and since $\Phi(0) = 0$}) \\
& < \int_{\Omega}\big(F(x,w_{\delta})\, dx - \int_{\Omega}f(x,w_{\delta})w_{\delta}\big)dx \leq 0,
\end{aligned}
\end{equation*}
 where we have used the growth condition \eqref{eq:growth_on_f}. 
 Summing up, we derive that $\mu_{\delta}\leq 0$ and that $w_{\delta} \in \mathcal{X}^{1,2}(\Omega)$ 
 is a of the Dirichlet problem
\begin{equation}
\begin{cases}
 \displaystyle\LL u = \frac{f(x,u)}{1-\mu_\delta} & \text{in $\Omega$}, \\
 u = 0 & \text{in $\R^n\setminus\Omega$}
\end{cases}
\end{equation}
 (in the sense of Definition \ref{def:weaksol}). Now, since $f$ satisfies
 the growth condition \eqref{eq:growth_on_f} and since $\mu_\delta\leq 0$,
 it is readily seen that the function
 \begin{equation} \label{eq:gdeltaDef}
  g_\delta(x,t) = \frac{f(x,t)}{1-\mu_\delta}\qquad(x\in\Omega,\,t\in\R)
  \end{equation}
 satisfies \emph{the same growth condition \eqref{eq:growth_on_f}}, with the same constant
 $C_f > 0$; we are then entitled to apply
 \cite[Theorem 1.1]{SVWZ2}, thus giving
 \begin{equation} \label{eq:stimaLinfVal}
  \|w_\delta\|_{L^\infty(\Omega)}\leq \kappa_0\Big(1+
   \int_\Omega|w_\delta|^{2^*\beta}\,dx\Big)^{\frac{1}{2^*(\beta-1)}},
 \end{equation}
 where $\beta = (2^*+1)/2$ and $\kappa_0 > 0$ is a constant independent of $\delta$.
 With \eqref{eq:stimaLinfVal} at hand, to complete the proof
 of \eqref{eq:wdeltaInfStepI} we need to show that
 \begin{equation} \label{eq:boundednessIntegral}
   \int_\Omega|w_\delta|^{2^*\beta}\,dx\leq c
\end{equation}  
 for some constant $c > 0$ independent of $\delta$, provided that $\delta$ is sufficiently small.
 To prove \eqref{eq:boundednessIntegral}, we closely
 follow the approach in \cite[Lemma 3.2]{SVWZ2}: defining
 \begin{equation*}
\varphi(t):= 
\begin{cases}
-\beta T^{\beta -1}(t+T) + T^{\beta}, & t \leq -T\\
|t|^{\beta}, & -T<t <T\\
\beta T^{\beta -1}(t-T) + T^{\beta}, & t \geq T.
\end{cases}
\end{equation*}
(with $T > 0$), and arguing exactly as in  \cite[Lemma 3.2]{SVWZ2}, we get
\begin{equation}\label{eq:3.5_Val2}
\begin{aligned}
& \Big(\int_\Omega|\varphi(w_\delta)|^{2^*}\,dx\Big)^{2/2^*} \leq c\beta
\Big(\int_{\Omega}|w_\delta|^{2^*}\, dx + 
  \int_{\Omega}(\varphi(w_\delta))^2|w_\delta|^{2^{\ast}-2}\, dx\Big),
\end{aligned}
\end{equation}
\noindent which is exactly  to \cite[Equation (3.5)]{SVWZ2} (here,
$c > 0$ is a constant only depending on $n,\Omega$ and on the constant $C_f$).
On the other hand, we have
\begin{align*}
 & \int_{\Omega}(\varphi(w_\delta))^2|w_\delta|^{2^{\ast}-2}\, dx \\
 & \qquad =
 \int_{\{|w_\delta|\leq 1\}}(\varphi(w_\delta))^2|w_\delta|^{2^{\ast}-2}\, dx
 + \int_{\{|w_\delta|>1\}}(\varphi(w_\delta))^2|w_\delta|^{2^{\ast}-2}\, dx \\
 & \qquad \leq 
 \int_{\{|w_\delta|\leq 1\}}\frac{\varphi(w_\delta)^2}{|w_\delta|}\, dx
 + \int_{\Omega}(\varphi(w_\delta))^2|w_\delta|^{2^{\ast}-2}\, dx = (\bigstar);
\end{align*}
then, using H\"older's and Sobolev's inequality, we obtain
\begin{align*}
 (\bigstar) & 
 \leq \int_{\{|w_\delta|\leq 1\}}\frac{|\varphi(w_\delta)|^2}{|w_\delta|}\, dx
 + \Big(\int_{\Omega}|\varphi(w_\delta)|^{2^*}\,dx\Big)^{2/2^*}
 \Big(\int_\Omega|w_\delta|^{2^{\ast}}\, dx\Big)^{\frac{2^*-2}{2^*}} \\
 & \leq \int_{\{|w_\delta|\leq 1\}}\frac{|\varphi(w_\delta)|^2}{|w_\delta|}\, dx
 + S_n^{\frac{2-2^*}{2}}\Big(\int_{\Omega}|\varphi(w_\delta)|^{2^*}\,dx\Big)^{2/2^*}
 \|w_\delta\|_{H^1_0(\Omega)}^{2^*-2} \\
 & (\text{since $\rho_\e(\cdot)\geq \|\cdot\|_{H^1_0}$, and since $w_\delta\in K_\delta$})\\
 & \leq \int_{\{|w_\delta|\leq 1\}}\frac{|\varphi(w_\delta)|^2}{|w_\delta|}\, dx
 + \delta^{2^*-2}S_n^{\frac{2-2^*}{2}}\Big(\int_{\Omega}|\varphi(w_\delta)|^{2^*}\,dx\Big)^{2/2^*}.
\end{align*}
where $S_n > 0$ is the best constant in the (local) Sobolev inequality. U\-sing this last estimate
in the above \eqref{eq:3.5_Val2}, we then derive 
\begin{equation}\label{eq:toreabsorb}
\begin{split}
  \Big(\int_\Omega|\varphi(w_\delta)|^{2^*}\,dx\Big)^{2/2^*} & \leq c\beta
\Big(\int_{\Omega}|w_\delta|^{2^*}\, dx + 
   \int_{\{|w_\delta|\leq 1\}}\frac{|\varphi(w_\delta)|^2}{|w_\delta|}\, dx\Big) \\
  & \quad
 + c\beta\delta^{2^*-2}S_n^{\frac{2-2^*}{2}}
 \Big(\int_{\Omega}|\varphi(w_\delta)|^{2^*}\,dx\Big)^{2/2^*}\Big).
\end{split}
\end{equation}
With \eqref{eq:toreabsorb} at hand, we can proceed
toward the end of the proof
of \eqref{eq:boundednessIntegral}: indeed, if we choose $\delta_0 \in (0,1)$ so small that 
$$c\beta\delta^{2^*-2}S_n^{\frac{2-2^*}{2}}<\frac{1}{2}$$
and if we let $0<\delta<\delta_0$, we can reabsorb the last integral in the right-hand side of 
the cited \eqref{eq:toreabsorb}, thus giving
\begin{equation*}
\begin{split}
  \Big(\int_\Omega|\varphi(w_\delta)|^{2^*}\,dx\Big)^{2/2^*} & \leq 2c\beta
\Big(\int_{\Omega}|w_\delta|^{2^*}\, dx + 
   \int_{\{|w_\delta|\leq 1\}}\frac{|\varphi(w_\delta)|^2}{|w_\delta|}\, dx\Big);
\end{split}
\end{equation*}
from this, by letting $T\to+\infty$ and by arguing as in
\cite[Lemma 3.2]{SVWZ2}, we get
\begin{align*}
 \Big(\int_\Omega|w_\delta|^{2^*\beta}\,dx\Big)^{2/2^*}
 \leq 4c\beta\int_{\Omega}|w_\delta|^{2^*}\, dx\quad\text{for every $0<\delta<\delta_0$}.
\end{align*}
Finally, by using once again the Sobolev inequality, we conclude that
\begin{align*}
 \Big(\int_\Omega|w_\delta|^{2^*\beta}\,dx\Big)^{2/2^*}
 \leq 4c\beta\,S_n^{-2^*/2}\rho_\e(w_\delta)\leq 4c\beta\,S_n^{-2^*/2},
\end{align*}
and this estimate holds for every $0<\delta<\delta_0$ (since $\delta_0<1$).
Recalling that $c > 0$
only depends on $n,\Omega$ and on the constant $C_f$,
this completes the proof of \eqref{eq:boundednessIntegral}.
\medskip

\textsc{Step II).} In this second step we complete the proof of
\eqref{eq:wdeltatozeroCLAIM}\,-\,i). To this end
it suffices to observe that, since the function 
$g_\delta$ in \eqref{eq:gdeltaDef} satisfies the growth condition
\eqref{eq:growth_on_f} \emph{with the same constant $C_f$} (independent of $\delta$),
we are entitled to apply
\cite[Theorem 1.3]{SVWZ2}; this, together with \eqref{eq:wdeltaInfStepI}, ensures that
\begin{itemize}
 \item[1)] $w_\delta\in C^{1,\alpha}(\overline{\Omega})$ for every $0<\alpha<\min\{1,2-2s\}$;
 \item[2)] there exists $c > 0$, independent of $\delta$, such that
 \begin{equation} \label{eq:uniformHolderwdelta}
  \|w_\delta\|_{C^{1,\alpha}(\overline{\Omega})}
  \leq c(1+\|w_\delta\|^{2^*}_{L^\infty(\Omega)})\leq c\quad\text{for every $0<\delta<\delta_0$}.
 \end{equation}
 \end{itemize}
 In particular, since $w_\delta\in\mathcal{X}^{1,2}(\Omega)$, we get $w_\delta\equiv 0$ pointwise
 on $\de\Omega$.
 \medskip
 
 \textsc{Step III).} In this last step we give the proof of \eqref{eq:wdeltatozeroCLAIM}\,-\,ii).
 To this end we o\-bserve that, on account of
 \eqref{eq:uniformHolderwdelta}, by the Arzel\`a-A\-sco\-li Theorem we can find
 a sequence $\{\delta_j\}\subseteq(0,\delta_0)$ converging to $0$ as $j\to+\infty$
 and $w_0\in C_0^1(\overline{\Omega})$ such that
 $$\|w_{\delta_j}- w_0\|_{C^1(\overline{\Omega})}\to 0\,\,\text{as $j\to+\infty$};$$
 thus, since $w_{\delta_j}\to 0$ in $\mathcal{X}^{1,2}(\Omega)$ as $j\to+\infty$ (recall that
 $w_{\delta}\in K_\delta$ and that the two norms $\rho_\e(\cdot),\,\rho(\cdot)$
 are globally equivalent in $\mathcal{X}^{1,2}(\Omega)$), we conclude that
 $$w_0\equiv 0.$$
 This ends the proof.
\end{proof}
\section{Proof of Theorem \ref{thm:main2}} \label{sec:proofMain2}
Following the approach in \cite{ABC}, in this section
we show how Theorem \ref{thm:H1vsC1} can be used in order
to establish Theorem \ref{thm:main2},
that is, the existence of (at least) \emph{two distinct solutions} of \eqref{eq:Problem}$_{\lambda}$
provided that $\e \in (0,1)$ is small enough.
\medskip

  We now turn to prove that, if $\e\in(0,1]$ is fixed
  and if $\lambda\in(0,\mu_\sharp)$ (with $\mu_\sharp > 0$ as in Lemma \ref{lem:existencemusharp}), 
  there exists a solution $u_{\lambda,\e}$ of problem \eqref{eq:Problem}$_{\lambda}$
  which is a \emph{local minimizer}
  for the functional associated with 
  the problem, that is,
  \begin{equation} \label{eq:funIlambdae}
   I_{\lambda,\e}(u) = \dfrac{1}{2}\mathcal{\rho}_{\varepsilon}(u)^2 - \dfrac{\lambda}{p+1}
  \int_{\Omega}|u|^{p+1}\, dx - \dfrac{1}{2^{\ast}}\int_{\Omega}|u|^{2^{\ast}}\, dx,
\end{equation}
\noindent where 
$
\mathcal{\rho}_{\varepsilon}(u)^2 := \|u\|_{H^1_0(\Omega)} + \varepsilon \, [u]^{2}_{s,\mathbb{R}^{n}}$.
\begin{theorem} \label{thm:ulambdaminimizer}
 Let $\e\in(0,1]$ be fixed, and let $\lambda\in(0,\mu_\sharp)$ \emph{(}with $\mu_\sharp > 0$ 
 as in Lemma \ref{lem:existencemusharp}\emph{)}. 
  Then, there exists a solution $u_{\lambda,\e}\in\mathcal{X}^{1,2}(\Omega)$
  of problem \eqref{eq:Problem}$_{\lambda}$ which is a local mi\-ni\-mi\-zer for $I_{\lambda,\e}$,
 that is, there exists $\varrho_0 > 0$ such that
 $$I_{\lambda,\e}(u)\geq I_{\lambda,\e}(u_{\lambda,\e})\quad\text{for all $u\in\mathcal{X}^{1,2}(\Omega)$
 with $\rho(u-u_{\lambda,\e}) < \varrho_0$}.$$
 Moreover, the following assertions hold:
 \begin{itemize}
  \item[{1)}] there exists a constant $C > 0$, independent of $\e$ but possibly 
  depending on the fixed $\lambda\in(0,\mu_\sharp)$, such that
  \begin{equation} \label{eq:uminlocLinf}
  \|u_{\lambda,\e}\|_{L^\infty(\Omega)}\leq C;
 \end{equation}
 \item[2)] if $w_{\lambda,\e}\in\mathcal{X}^{1,2}(\Omega)$ is the unique solution
 of \eqref{eq:SublinearProblem}, we have
 \begin{equation} \label{eq:lowerboundWMP}
  u_{\lambda,\e}\geq w_{\lambda,\e}\quad\text{a.e.\,in $\Omega$}.
 \end{equation}
 \end{itemize}
\end{theorem}
\begin{proof}
  The proof is analogous to that of \cite[Lemma 4.1]{ABC}, but we
  present it here for the sake of completeness. To begin with, we fix $\lambda_1,\lambda_2
  \in (0,\mu_\sharp)$ such that 
  \begin{equation} \label{eq:choicelambda12}
   \lambda_1<\lambda<\lambda_2;
  \end{equation} accordingly, we let
  $u_i = \bar u_{\lambda_i,\e}\in\mathcal{X}^{1,2}(\Omega)$ be the minimal solution of 
  problem \eqref{eq:Problem}$_{\lambda_i}$ (with $i = 1,2$), whose existence is guaranteed
  by Lemma \ref{lem:existencemusharp}.
  We then observe that, on account of Theorem 
  \ref{thm:regulAntCozzi}, we have that $u_1,u_2\in C^{1}(\overline{\Omega})$; 
  moreover, since the map $t\mapsto \bar u_{t,\e}$ is non-decreasing on $(0,\Lambda_\e)$,
  by \eqref{eq:choicelambda12} we also have 
  $$\text{$u_1\leq u_2$ pointwise in $\Omega$}.$$
  In view of these facts, and since $u_i$ solves 
 \eqref{eq:Problem}$_{\lambda_i}$ (for $i = 1,2$), we infer that
  \begin{itemize}
  \item[a)] $w = u_2-u_1\in\mathcal{X}^{1,2}(\Omega)\cap C^1(\overline{\Omega})$;

  \item[b)] for every test function 
  $\varphi\in C_0^\infty(\Omega),\,\text{$\varphi\geq 0$ on $\Omega$}$, we have
  \begin{align*}
  \mathcal{B}_{\varepsilon}(w,\varphi) & = \int_\Omega(\lambda_2 u_2^p+u_2^{2^*-1})\varphi\,dx
   -\int_\Omega(\lambda_1 u_1^p+u_1^{2^*-1})\varphi\,dx \\
   & (\text{by \eqref{eq:choicelambda12}, and since $u_2\geq u_1\geq 0$}) \\
   & \geq \lambda_1\int_\Omega(u_2^p-u_1^p)\varphi\,dx+
   \int_\Omega(u_2^{2^*-1}-u_1^{2^*-1})\varphi\,dx \geq 0
  \end{align*}
  and this proves that $w$ is a
  (non-negative) \emph{weak su\-per\-so\-lution} (in the sense of
  Definition \ref{def:weaksol}) of
  the following problem
  $$\begin{cases}
   \LL u = 0 & \text{in $\Omega$} \\
   u = 0 & \text{in $\R^n\setminus\Omega$}
  \end{cases}
  $$   
 \end{itemize} 
  Since $w\not\equiv 0$ in $\Omega$ (as $\lambda_1<\lambda_2$), by
  Theorem \ref{thm:WMPSMPHopfTogether} we conclude that
  \begin{equation} \label{eq:HopfandSMPv}
   \text{$w = u_2 - u_1 > 0$ pointwise in $\Omega$}\quad\text{and}\quad
   \text{$\de_{\nu}(u_2-u_1) < 0$ on $\de\Omega$}.
  \end{equation}
  With \eqref{eq:HopfandSMPv} at hand, we now define the function
  $$
  f(x,t) = \begin{cases}
  \lambda u_1(x)^p+u_1(x)^{2^*-1} & \text{if $t \leq u_1(x)$} \\
  \lambda t^p+t^{2^*-1} & \text{if $u_1(x)<t<u_2(x)$} \\
  \lambda u_2(x)^p+u_2(x)^{2^*-1} & \text{if $t\geq u_2(x)$}
  \end{cases}$$
  (for $x\in\Omega$ and $t\in\R$), and we consider the functional
  $$J(u) = \frac{1}{2}\rho_\e(u)^2-\int_\Omega F(x,u)\,dx,\qquad
  \text{with $F(x,t) = \int_0^t f(x,s)\,ds$}.$$
   Since $u_1,u_2\in C^1(\overline{\Omega})$,
   the function $f$ is \emph{globally bound\-ed} on $\Omega\times\R$;
  as a consequen\-ce of this fact, it is easy to check that the functional $J$
  achieves its (global) minimum at some $u\in\mathcal{X}^{1,2}(\Omega)$,
  which is a weak solution of 
  \begin{equation} \label{eq:PbDirminu}
  \begin{cases}
  \LL u = f & \text{in $\Omega$}, \\
  u = 0 & \text{on $\de\Omega$}
  \end{cases}
  \end{equation}
  Moreover, from Theorem \ref{thm:regulAntCozzi}
  we get $u\in C^{1}(\overline{\Omega})$.
  We then claim that
   \begin{equation} \label{eq:Claimuu1u2}
   \begin{split}
    \mathrm{i)}\,\,&\text{$u_1<u<u_2$ pointwise in $\Omega$}; \\
    \mathrm{ii)}\,\,&\text{$\de_{\nu}(u-u_1) < 0$ and $\de_{\nu}(u-u_2) > 0$ pointwise on $\de\Omega$}.
   \end{split}
   \end{equation}
 In fact, let $w_1 = u-u_1$. On the one hand, we have
  $w_1\in\mathcal{X}^{1,2}(\Omega)\cap C^{1}(\overline{\Omega})$
 (as the same is true of both $u$ and $u_1$);
 on the other hand, since 
 $u_1$ solves $\mathrm{(P)_{\lambda_1,\e}}$ and since $u$ solves
 \eqref{eq:PbDirminu}, by definition of $f$ (and by recalling \eqref{eq:HopfandSMPv}) we also have
 \begin{align*}
   \mathcal{B}_{\varepsilon}(w_1,\varphi) & = 
   \int_\Omega\big(f(x,u)- \lambda u_1^p-u_1^{2^*-1}\big)\varphi\,dx 
   \geq 0
 \end{align*}
 for every test function $\varphi\in C_0^\infty(\Omega),\,\text{$\varphi\geq 0$ on $\Omega$}$.
 Hence, we see that $w_1$ is a weak supersolution
 (in the sense of Definition \ref{def:weaksol}) 
 of the problem
 \begin{equation*}
  \begin{cases} 
  \LL u = 0 & \text{in $\Omega$}, \\
  u = 0 & \text{in $\R^n\setminus\Omega$}
  \end{cases}
 \end{equation*}
  Since $w_1\not\equiv 0$ in $\Omega$ (as
  $\lambda_1<\lambda)$, again by Theorem \ref{thm:WMPSMPHopfTogether} we then get
  $$\text{$w_1 = u-u_1 > 0$ pointwise in $\Omega$}\quad\text{and}\quad
   \text{$\de_{\nu}(u-u_1) < 0$ on $\de\Omega$}.$$
   In a totally analogous way one can prove that $w_2 = u_2-u > 0$ pointwise in $\Omega$ and that
   $\de_{\nu}w_2 < 0$ on $\de\Omega$, and this completes the proof of \eqref{eq:Claimuu1u2}.
   \vspace*{0.1cm}
   
   Now we have established \eqref{eq:Claimuu1u2}, we are finally
   ready to complete the demonstration of the theorem. Indeed, owing to \eqref{eq:Claimuu1u2},
    there exists $\varrho_0 > 0$ so small that
    $$\|v-u\|_{C^1(\overline{\Omega})} < \varrho_0\,\,\Longrightarrow\,\,\text{$u_1\leq v\leq u_2$ pointwise
    in $\Omega$};$$
    from this, taking into account that $J(v) = I_{\lambda,\e}(v)$ for every $v\in\mathcal{X}^{1,2}(\Omega)$
    satisfying $u_1\leq v\leq u_2$ in $\Omega$ (see the explicit
    definition of $f$), we derive that
    $$I_{\lambda,\e}(v) = J(v) \geq J(u) = I_{\lambda,\e}(u) \quad
    \text{for all $v\in C^1(\overline{\Omega})$ with $\|v-u\|_{C^1(\overline{\Omega})}<\varrho_0$},$$
    and this proves that $u$ is a local minimizer of $I_{\lambda,\e}$ in the $C^1$-topology.
    This, together with Theorem \ref{thm:H1vsC1}, ensures that 
    $u\in \mathcal{X}^{1,2}(\Omega)\cap C^{1}(\overline{\Omega})$ is a local minimizer of 
    the functional $I_{\lambda,\e}$ also in the $\mathcal{X}^{1,2}$-topology, as desired.
    \vspace*{0.1cm}
    
    \vspace*{0.1cm}
    
    We now turn to prove the two assertions 1)\,-\,2).
    \medskip
    
    \noindent-\,\,\emph{Proof of assertion} 1). It is a direct consequence of
    \eqref{eq:Claimuu1u2}\,-\,i)
    and of estimate
    \eqref{eq:minimalLinf} in Lem\-ma
    \ref{lem:existencemusharp}, taking into account that 
     $u_1 = \bar{u}_{\lambda_1,\e}$.
     \medskip
     
     \noindent-\,\,\emph{Proof of assertion} 2). 
     First of all we observe that, since $u_{\lambda,\e}$ is a weak solution of
     problem \eqref{eq:Problem}$_\lambda$ (hence, in particular,
     $u_{\lambda,\e} > 0$ a.e.\,in $\Omega$), this function is a 
     \emph{weak supersolution} of  problem \eqref{eq:SublinearProblem} (in
     the sense of Definition \ref{def:weaksol} with $f(t) = \lambda t^p$);
     thus, since $w_{\lambda,\e}$ solves
     \eqref{eq:SublinearProblem}, we can use \cite[Lemma 3.11]{AMT}
     (which can certainly be applied
     to the operator $\LL$, the role of $\e$ being immaterial), obtaining
     $$w_{\lambda,\e}\leq u_{\lambda,\e}\quad\text{a.e.\,in $\Omega$}.$$
     We explicitly mention that \cite[Lemma 3.11]{AMT} can be applied, since
     $f(t)/t = \lambda t^{p-1}$ is decreasing, and since
     $w_{\lambda,\e}\,\,u_{\lambda,\e}\in L^{2^*}(\Omega)$, see Remark 
     \ref{rem:spaceX12}\,-\,4).
     \medskip
     
     \noindent This ends the proof.
\end{proof}
Now we have established Theorem \ref{thm:ulambdaminimizer}, 
we are ready to embark on the demonstration of Theorem \ref{thm:main2}. To keep the notation
as simple as possible, throughout what follows we avoid to keep explicit track of the dependence on
$\e$ of the functional $I_{\lambda,\e}$ in \eqref{eq:funIlambdae}
and of the solution $u_{\lambda,\e}$ obtained in Theorem \ref{thm:ulambdaminimizer}. Thus, as in the
previous sections we simply write $I_\lambda$ and $u_\lambda$,
the dependence on $\e$ being understood.
\medskip

We begin with the following key lemma.
\begin{lemma}\label{lem:CrucialLemma}
Let $\lambda \in (0,\mu_{\sharp})$ be fixed \emph{(}with $\mu_\sharp>0$ as in \eqref{eq:existencemusharp}\emph{)}.
Then, there exist $\varepsilon_{0} \in (0,1)$, $R_{0}>0$ and a positive function $\Psi\in\mathcal{X}^{1,2}(\Omega)$ such that
\begin{equation} \label{eq:CrucialLemmaEq}
\left\{ \begin{array}{lr}
I_{\lambda}(u_{\lambda}+R\Psi) < I_{\lambda}(u_{\lambda}) & \textrm{for all  $\varepsilon \in (0,\varepsilon_{0})$ and $R\geq R_{0}$},\\
I_{\lambda}(u_{\lambda}+t R_{0}\Psi) < I_{\lambda}(u_{\lambda}) + \tfrac{1}{n} S_{n}^{n/2} & \textrm{for all  $\varepsilon \in (0,\varepsilon_{0})$ and $t \in [0,1]$}.
\end{array}\right.
\end{equation}
\begin{proof}
  First of all, we 
  choose a \emph{Lebesgue point} of $u_\lambda$ in $\Omega$, say $y$, 
  and we let $r > 0$ be such that $B_r(y)\Subset\Omega$;
  we then choose a cut-off function $\varphi\in C_0^\infty(\Omega)$ such that 
  \begin{itemize}
   \item[$(\ast)$] $0\leq\varphi\leq 1$ in $\Omega$;
   \item[$(\ast)$] $\varphi\equiv 1$ on $B_r(y)$;
  \end{itemize}
  and we consider the one-parameter family of functions
  \begin{equation}\label{eq:Scelta_Talenti}
  U_\e = V_\e\,\varphi,\qquad
  \text{where}\,\,V_\e(x) = \frac{\e^{\frac{\alpha(n-2)}{2}}}{(\e^{2\alpha}+|x-y|^2)^{\frac{n-2}{2}}},
  \end{equation}
  where $\alpha > 0$ will be appropriately chosen later on.
  Notice that this family $\{V_\e\}_\e$
  is the well-kno\-wn family of the \emph{Aubin-Ta\-len\-ti functions}, which are
  the unique (up to translation) extremals in the (local) Sobolev inequality; 
  this means that
  \begin{equation} \label{eq:Veminim}
   \frac{\||\nabla V_\e|\|^2_{L^2(\R^n)}}{\|V_\e\|^2_{L^{2^*}(\R^n)}} = S_n,
  \end{equation}
  where $S_n > 0$ is the \emph{best constant} in the Sobolev inequality.   
  
  \noindent We now recall that, owing to \cite[Lemma 1.1]{BN}, we have (as $\e \to 0^+$)
  \begin{equation} \label{eq:estimVeLocalBN}
   \begin{split}
    \mathrm{i)}&\,\,\| U_\e\|^2_{H^1_0(\Omega)} = 
    \e^{\alpha(n-2)}\Big(\frac{K_1}{\e^{\alpha(n-2)}}+O(1)\Big)
    = K_1+o(\e^{\frac{\alpha(n-2)}{2}}); \\[0.1cm]
    \mathrm{ii)}&\,\,\|U_\e\|^{2^*}_{L^{2^*}(\Omega)} = \e^{\alpha n}\Big(\frac{K_2}{\e^{\alpha n}}+O(1)\Big)
    = K_2+o(\e^{\alpha n/2});
   \end{split}
  \end{equation}
  where the constants $K_1,K_2 > 0$ are given, respectively, by
  $$K_1 = \||\nabla V_1|\|^2_{L^2(\R^n)},\qquad K_2 = \int_{\R^n}\frac{1}{(1+|x|^2)^n}\,dx$$
  and they satisfy 
  $K_1/K_2^{1-2/n} = S_n$, see the above \eqref{eq:Veminim}.
  On the other hand, by combining
   \eqref{eq:estimVeLocalBN}\,-\,ii) with
   \cite[Lemma 4.11]{BDVV5},
  we also have 
  \begin{equation} \label{eq:estimVeNonlocal}
  \begin{split}
   \e \, [U_\e]^2_s  & \leq \e \, \|U_\e\|^2_{L^{2^*}(\Omega)}\cdot O(\e^{\alpha(2-2s)}) = \|U_\e\|^2_{L^{2^*}(\Omega)}\cdot O(\e^{1+\alpha(2-2s)})
   \\
   & = \big(K_2+o(\e^{\alpha n/2})\big)^{2/2^*}\cdot O(\e^{1+\alpha(2-2s)}) = o(\e^{1+\alpha(1-s)}).
   \end{split}
  \end{equation}
 Defining
  $w = u_\lambda + tRU_\e$ (for $R\geq 1$ and $t\in[0,1]$),
  we then obtain the following estimate (as $\e\to 0^+$):
  \begin{equation*}
  \begin{split}
   I_{\lambda,\e}(w) & = I_{\lambda,\e}(u_\lambda) 
   + \frac{t^2R^2}{2}\rho_\e(U_\e)^2+tR\mathcal{B}_\rho(u_\lambda,U_\e)
   \\
   & \qquad
   - \frac{1}{2^*}\big(\|u_\lambda+tRU_\e\|^{2^*}_{L^{2^*}(\Omega)}-
    \|u_\lambda\|^{2^*}_{L^{2^*}(\Omega)}\big) \\
   &\qquad
   -\frac{\lambda}{p+1}\Big(\int_\Omega|u_\lambda+tRU_\e|^{p+1}\,dx
   - \int_\Omega|u_\lambda|^{p+1}\,dx\Big),
   \end{split}
   \end{equation*}
   \noindent and recalling that $u_\lambda$ solves \eqref{eq:Problem}$_{\lambda}$, we get
   \begin{equation*}
   \begin{split}
   I_{\lambda,\e}(w) & = I_{\lambda,\e}(u_\lambda)
   + \frac{t^2R^2}{2}\rho_\e(U_\e)^2
   +tR\int_\Omega(\lambda u_\lambda^{p}+u_\lambda^{2^*-1})U_\e\,d x
   \\
   & \qquad
   - \frac{1}{2^*}\big(\|u_\lambda+tRU_\e\|^{2^*}_{L^{2^*}(\Omega)}-
    \|u_\lambda\|^{2^*}_{L^{2^*}(\Omega)}\big) \\
   &\qquad
   -\frac{\lambda}{p+1}\Big(\int_\Omega|u_\lambda+tRU_\e|^{p+1}\,dx
   - \int_\Omega|u_\lambda|^{p+1}\,dx\Big) \\
   & = I_{\lambda,\e}(u_\lambda)
   + \frac{t^2R^2}{2}\rho_\e(U_\e)^2
   - \frac{t^{2^*}R^{2^*}}{2^*}\|U_\e\|^{2^*}_{L^{2^*}(\Omega)} \\
   & \qquad - t^{2^*-1}R^{2^*-1}\int_\Omega U_\e^{2^*-1}u_\lambda\,dx
   - \mathcal{R}_\e-\mathcal{D}_\e 
\end{split}
\end{equation*}   
      where we have introduced the notation
   \begin{align*}
    (\ast)&\,\,\mathcal{R}_\e = \frac{1}{2^*}\int_\Omega\Big\{|u_\lambda+tRU_\e|^{2^*}
    - u_\lambda^{2^*}-(tRU_\e)^{2^*}
    \\
    & \qquad\qquad\qquad 
     - 2^* u_\lambda (tRU_\e)\big(u_\lambda^{2^*-2}+ (tRU_\e)^{2^*-2}\big)\Big\}dx; \\
     (\ast)&\,\,\mathcal{D}_\e = 
     \frac{\lambda}{p+1}\int_\Omega
     \Big\{|u_\lambda+tRU_\e|^{p+1}- |u_\lambda|^{p+1}-
     tR(p+1)u_\lambda^{p}U_\e\Big\}dx.
   \end{align*}

   We start estimating $\mathcal{R}_{\e}$. To this aim, we follow \cite[Proof of Theorem 1]{BN89} (from equation (17) on) where the main difference is due to the fact that $u_{\lambda}$ is actually dependent on $\e$ as well. However, using the uniform upper bound of the $L^{\infty}$-norm of $u_{\lambda}$
   in Theorem \ref{thm:ulambdaminimizer}, one can get
      \begin{equation} \label{eq:termRe}
    |-\mathcal{R}_\e| \leq R^\beta o(\e^{\frac{\alpha(n-2)}{2}})\quad\text{as $\e\to 0^+$},
   \end{equation}
   for some $\beta\in(0,2^*)$. Regarding $\mathcal{D}_{\e}$, since the map
   $$[0,+\infty)\ni r\mapsto \gamma(r) = r^{p+1}$$
   is \emph{convex on $[0,+\infty)$}, for every $\e\in(0,1)$ we have
   \begin{equation} \label{eq:termDe}
   \begin{split}
    -\mathcal{D}_\e & = \frac{\lambda}{p+1}\int_\Omega
     \Big\{tR(p+1)u_\lambda^{p}U_\e+|u_\lambda|^{p+1}-|u_\lambda+tRU_\e|^{p+1} \\
     & = -\frac{\lambda}{p+1}\int_\Omega
     \Big\{\gamma(u_\lambda+tRU_\e)-\gamma(u_\lambda)
     - \gamma'(u_\lambda)(tRU_\e)
     \Big\}dx \leq 0.
     \end{split}
   \end{equation}
    Finally, combining once again
   \eqref{eq:lowerboundWMP} and Proposition \ref{prop:SublinearGeneral}\,-\,iii), and arguing as in
    the proof of \cite[Lemma 4.11]{BDVV5}, we get
   \begin{equation}\label{eq:termUeTarantello}
   \begin{aligned}
     \int_\Omega U_\e^{2^*-1}u_\lambda\,dx &\geq 
     \int_{B_{r}(y)}V_{\e}^{2^*-1}w_{\lambda,\e}\,dx \geq c_2 \, \int_{B_{r}(y)}V_{\e}^{2^*-1}\,dx \\
     &\geq
     \e^{\alpha(n-2)/2}D_0 +
     o(\e^{\alpha(n-2)/2})\quad\text{as $\e\to 0^+$},
     \end{aligned}
   \end{equation}
 where $D_0 > 0$ is a constant only depending on $n$.
  
  Gathering \eqref{eq:estimVeLocalBN}-\eqref{eq:estimVeNonlocal}, we obtain   
 \begin{equation}
   \begin{aligned}
   & I_{\lambda,\e} \leq I_{\lambda,\e}(u_\lambda)+\frac{t^2R^2}{2}K_1-
   \frac{t^{2^*}R^{2^*}}{2^*}K_2 \\
   &\qquad - t^{2^*-1}R^{2^*-1} D_0 \e^{\alpha(n-2)/2}
   - R^{\beta} o(\e^{\alpha(n-2)/2})+ C(tR+t^2R^2)\,o(\e^{\frac{\alpha(n-2)}{2}})
   \\
   &\qquad\qquad +\big(t^2R^2+t^{2^*}R^{2^*}\big)\big(o(\e^{1+\alpha(1-s)})+o(\e^{\alpha(n-2)/2})\big).
   \end{aligned}
\end{equation}  
   Thus, if we choose $\alpha\in(0,1]$ so small that  
   $1+\alpha(1-s)>\alpha(n-2)/2$
   (notice that this is always possible, since $n\geq 3$ and $s\in(0,1)$), we get
   \begin{equation} \label{eq:toconcludecomeinTar}
   \begin{split}
    I_{\lambda,\e}(u_\lambda+tRU_\e) & \leq I_{\lambda,\e}(u_\lambda)+\frac{t^2R^2}{2}K_1-
   \frac{t^{2^*}R^{2^*}}{2^*}K_2 \\
   &\qquad - t^{2^*-1}R^{2^*-1}D_0 \, \varepsilon^{\alpha(n-2)/2}\\   
   &\qquad\qquad +C \big(t^2R^2+t^{2^*}R^{2^*}\big)o(\e^{\alpha(n-2)/2}).
    \end{split}
   \end{equation}
   Now, thanks to estimate \eqref{eq:toconcludecomeinTar} we are finally
   ready to complete the proof of the lemma: in fact, starting from this estimate
   and repeating word by word the argument in \cite[Lemma 3.1]{Tarantello}, we find
   $\e_0 > 0$ and $R_0 > 0$ such that
   \begin{equation*}
    \begin{cases}
     I_\lambda(u_\lambda+RU_\e) < I_\lambda(u_\lambda) & \text{for all $\e\in(0,\e_0)$ and $R\geq R_0$}, \\
      I_\lambda(u_\lambda+tR_0 U_\e) < I_\lambda(u_\lambda)+\frac{1}{n}S_n^{n/2}
      & \text{for all $\e\in(0,\e_0)$ and $t\in[0,1]$}.
    \end{cases}
   \end{equation*}
   Thus, the lemma is proved by choosing $\Psi = U_\e$ (with $\e < \e_0$ and $y,a$ as above).
\end{proof}
\end{lemma}
Thanks to Lemma \ref{lem:CrucialLemma}, we can now proceed toward the proof
of Theorem \ref{thm:main2}: in fact, we turn to show that problem
\eqref{eq:Problem}$_\lambda$ possesses a second solution $v_\lambda\neq u_\lambda$, provided
that $0<\lambda<\lambda_*$ and $0<\e<\e_0$, \emph{where $\e_0 > 0$ is as in Lemma \ref{lem:CrucialLemma}}.

Let then $0<\e<\e_0$ be arbitrarily but fixed (with $\e_0 > 0$ as in Lemma
\ref{lem:CrucialLemma}), and let $0<\lambda <\lambda_*$. Since we know that $u_\lambda$
is a local minimizer of the functional $I_\lambda$,
there exists some $0<\varrho_0 = \varrho_0(\e) \leq \rho_\e(u_\lambda)$ such that
 \begin{equation} \label{eq:ulambdaminEkeland}
  I_\lambda(u)\geq I_\lambda(u_\lambda)\quad\text{for every $u\in\mathcal{X}^{1,2}(\Omega)$ with
 $\rho_\e(u-u_\lambda)<\varrho_0$}.
 \end{equation}
 As a consequence of \eqref{eq:ulambdaminEkeland}, if we consider the cone 
 \begin{equation*}
  T = \{u\in\mathcal{X}^{1,2}(\Omega):\,\text{$u\geq u_\lambda>0$ a.e.\,in $\Omega$}\},
 \end{equation*}
 only one of the following two cases hold:
 \begin{itemize}
  \item[\textsc{A)}] $\inf\{I_\lambda(u):\,\text{$u\in T$ and $\rho_\e(u-u_\lambda) = \varrho$}\}
  = I_\lambda(u_\lambda)$ for every $0<\varrho<\varrho_0$;
  \vspace*{0.05cm}
  \item[\textsc{B)}] there exists $\varrho_1\in(0,\varrho_0)$ such that 
  \begin{equation} \label{eq:defvarrho}
   \inf\{I_\lambda(u):\,\text{$u\in T$ and $\rho_\e(u-u_\lambda) = \varrho_1$}\} > I_\lambda(u_\lambda).
   \end{equation}
 \end{itemize}
 We explicitly notice that use of $\rho_\e(\cdot)$ to define an open neighborhood
 of $u_\lambda$ in \eqref{eq:ulambdaminEkeland} is motivated by the fact that
 the norm $\rho_\e(\cdot)$ is globally equivalent to $\rho(\cdot)$ (and to the 
 $H_0^1$-norm), uniformly w.r.t.\,$\e$, see \eqref{eq:equivalenceuniforme}.
 \medskip

 We then turn to consider the two cases A)\,-\,B) separately. In doing this,
 we will re\-pea\-tedly use the following elementary result of Real Analysis.
 \begin{lemma} \label{lem:Convergence}
  Let $A\subseteq\R^n$ be an arbitrary measurable set, and let $1\leq m< \infty$.
  Moreover, let $f\in L^m(A)$ and let $\{f_j\}_j\subseteq L^m(A)$ be such that
  $$\text{$f_j\to f$ in $L^m(A)$ as $j\to+\infty$}.$$
  Then, for every $0<\vartheta\leq m$ we have
  \begin{equation} \label{eq:LemmaConv}
   \lim_{j\to+\infty}\int_A |f_j|^\vartheta\varphi\,dx = \int_A |f|^\vartheta\varphi\,dx\quad
  \text{for every $\varphi\in L^{m'_\vartheta}(A)$},
\end{equation}   
where $m'_\vartheta = m/(m-\vartheta)$ 
\emph{(}with the convention $m'_\vartheta = \infty$ if $\vartheta = m$\emph{)}.
 \end{lemma}
 \begin{proof}
  We preliminary observe that, taking into account the very definition of $m'_\vartheta$,
 a direct application of H\"older's inequality shows that
 \begin{equation} \label{eq:uvL1prel}
  |u|^\vartheta |v|\in L^1(A)\quad\text{for every $u\in L^m(A)$ and every $v\in L^{m'_\vartheta}(A)$};
 \end{equation}
 hence, all the integrals in \eqref{eq:LemmaConv} are well-defined and finite. 
 \vspace*{0.05cm}
 
 With \eqref{eq:uvL1prel} at hand, we now turn to establish \eqref{eq:LemmaConv}.
 To this end, we arbitrarily choose a sub-sequence $\{f_{j_k}\}_k$ of
 $\{f_j\}_j$ and we prove that, by possibly choosing a fur\-ther sub-sequence, 
 identity \eqref{eq:LemmaConv} holds for $\{f_{j_k}\}_k$.
 
 Let then $\varphi\in L^{m'_\vartheta}(A)$ be fixed. Since, by assumption, $f_j\to f$ strongly in $L^m(A)$
 as $j\to+\infty$, we can find a function $h\in L^m(A)$ such that
 (up to a sub-sequence)
 \begin{itemize}
  \item[i)] $f_{j_k}\to f$ a.e.\,in $A$ as $k\to+\infty$;
  \vspace*{0.05cm}
  
  \item[ii)] $0\leq |f_{j_k}|\leq h$ a.e.\,in $A$, for every $k\in\N$;
 \end{itemize}
 thus, since $\vartheta > 0$, by assertion ii) we have the estimate
 \begin{equation} \label{eq:estimfjkh}
  \text{$0\leq |f_{j_k}|^\vartheta|\varphi|\leq h^\vartheta|\varphi|$ a.e.\,in $A$ and for every $k\in\N$}.
 \end{equation}
 Now, by \eqref{eq:uvL1prel} we know that $g =  h^\vartheta|\varphi|\in L^1(A)$; this,
 together with assertion i) and estimate \eqref{eq:estimfjkh}, allows us to apply the Lebesgue 
 Theorem, giving
 $$\lim_{k\to+\infty}\int_A |f_{j_k}|^\vartheta\varphi\,dx = \int_A |f|^\vartheta\varphi\,dx.$$
 This ends the proof.
 \end{proof}
 We are now ready to prove the following propositions.
 \begin{proposition} \label{prop:Lemma26Haitao}
  Assume that 
  \textsc{Case A)} holds. Then, for e\-ve\-ry $\varrho\in(0,\varrho_0)$ there exists
  a solution $v_\lambda$ of problem \eqref{eq:Problem}$_\lambda$ such that
  $$\text{$\rho_\e(u_\lambda-v_\lambda) = \varrho$}.$$
  In particular, $v_\lambda\not\equiv u_\lambda$.
 \end{proposition}
 \begin{proof}
  Let $0<\varrho<\varrho_0$ be arbitrarily fixed. Since we are assuming that \textsc{Case A)} holds,
  we can find a sequence $\{u_k\}_k\subseteq T$ satisfying the following properties:
  \begin{itemize}
   \item[a)] $\rho_\e(u_k-u_\lambda) = \varrho$ for every $k\geq 1$;
   \vspace*{0.05cm}
   
   \item[b)] $I_\lambda(u_k)\to I_\lambda(u_\lambda) =:\mathbf{c}_\lambda$ as $k\to+\infty$.
  \end{itemize}
  We then choose $\delta > 0$ so small that $\varrho-\delta > 0$ and $\varrho+\delta<\varrho_0$ and,
  accordingly, we consider the subset of $T$ defined as follows:
  $$X = \{u\in T:\,\varrho-\delta\leq\rho_\e(u-u_\lambda)\leq \varrho+\delta\}
  \subseteq T$$
  (note that $u_k\in X$ for every $k\geq 1$, see a)). Since it is \emph{closed}, this set $X$
  is a \emph{complete metric space} when endowed with the distance induced by $\rho$;
  moreover, since $I_\lambda$ is a \emph{real-valued and continuous functional} on $X$, and since
  $$\textstyle\inf_{X}I_\lambda = I_\lambda(u_\lambda)$$
  we are entitled to apply
  the Ekeland Variational Principle (see \cite{AubEke}) to the functional $I_\lambda$ on $X$,
  providing us with a sequence $\{v_k\}_k\subseteq X$ such that
  \begin{equation} \label{eq:EkelandCaseA}
  \begin{split}
    \mathrm{i)}&\,\,I_\lambda(v_k)\leq I_\lambda(u_k)\leq I_\lambda(u_\lambda)+1/k^2, \\
    \mathrm{ii)}&\,\,\rho_\e(v_k-u_k)\leq 1/k, \\
    \mathrm{iii)}&\,\,I_\lambda(v_k)\leq I_\lambda(u)+1/k\cdot
     \rho(v_k-u)\quad\text{for every $u\in X$}.    
    \end{split}
  \end{equation}
  We now observe that, since $\{v_k\}_k\subseteq X$ and since the set $X$ is \emph{bounded}
 in $\mathcal{X}^{1,2}(\Omega)$, 
  there exists  $v_\lambda\in\mathcal{X}^{1,2}(\Omega)$ such that (as $k\to+\infty$ and 
  up to a sub-sequence)
  \begin{equation} \label{eq:limitvkCaseA}
  \begin{split}
   \mathrm{i)}&\,\,\text{$v_k\to v_\lambda$ weakly in $\mathcal{X}^{1,2}(\Omega)$}; \\
   \mathrm{ii)}&\,\,\text{$v_k\to v_\lambda$ strongly in $L^m(\Omega)$ for every $1\leq m<2^*$}; \\
   \mathrm{iii)}&\,\,\text{$v_k\to v_\lambda$ pointwise
   a.e.\,in $\Omega$}.
   \end{split}
  \end{equation}
  where we have also used the \emph{compact embedding} 
  $\mathcal{X}^{1,2}(\Omega)\hookrightarrow L^2(\Omega)$. 
  To complete the proof, we then turn to prove the following two facts:
  \begin{itemize}
   \item[1)] $v_\lambda$ is a weak solution of \eqref{eq:Problem}$_\lambda$;
   \item[2)] $\rho_\e(v_\lambda-u_\lambda) = \varrho > 0$.
  \end{itemize}
  \vspace*{0.1cm}
  
  \noindent \emph{Proof of} 1). To begin with, we fix
  $w\in T$ and we choose $\nu_0 = \nu_0(w,\lambda) > 0$ so small that 
   $u_\nu = v_k+\nu(w-v_k) \in X$ for every $0<\nu<\nu_0$.
   We explicitly stress that the existence of such an $\nu_0$ easily follows
   from \eqref{eq:EkelandCaseA}-ii) and property a) above.
   \vspace*{0.05cm}
   
   On account of \eqref{eq:EkelandCaseA}-iii)
   (with $u = u_\nu$), we have
   $$\frac{I_\lambda(v_k+\nu(w-v_k))-I_\lambda(v_k)}{\nu}
   \geq -\frac{1}{k}\rho_\e(w-v_k);$$
   From this, by letting $\nu\to 0^+$ we obtain
   \begin{equation} \label{eq:221Haitao}
   \begin{split}
    -\frac{1}{k}\rho_\e(w-v_k) & \leq {\mathcal{B}_{\e}}(v_k,w-v_k)
    - \int_\Omega v_k^{2^*-1}(w-v_k)\,dx \\
    & \qquad-\lambda\int_\Omega v_k^{p}(w-v_k)\,dx,
   \end{split}
   \end{equation}
   Now, given any $\varphi\in C_0^\infty(\Omega)$ and any $\nu > 0$, we define
   \begin{itemize}
    \item[$(\ast)$] $\psi_{k,\nu} = v_k+\nu\varphi-u_\lambda$ and $\phi_{k,\nu} = (\psi_{k,\nu})_-$;
    \vspace*{0.05cm}
    
    \item[$(\ast)$] $\psi_\nu = v_\lambda+\nu\varphi-u_\lambda$ and $\phi_\nu = (\psi_\nu)_-$.
   \end{itemize}
   Since, obviously, $w = v_k+\nu\varphi+\phi_{k,\nu}\in T$,
   by exploiting \eqref{eq:221Haitao} we get
   \begin{equation} \label{eq:topasstothelimitCaseASol}
    \begin{split}
	-\frac{1}{k}\rho_\e(\nu\varphi+\phi_{k,\nu}) & \leq {\mathcal{B}_{\e}}(v_k,\nu\varphi+\phi_{k,\nu})
    - \int_\Omega v_k^{2^*-1}(\nu\varphi+\phi_{k,\nu})\,dx \\
    & \qquad-\lambda\int_\Omega v_k^{p}(\nu\varphi+\phi_{k,\nu})\,dx.
    \end{split}
   \end{equation}
   Then, we aim to pass to the limit as $k\to+\infty$ and $\nu\to 0^+$ in the above 
   \eqref{eq:topasstothelimitCaseASol}. To this end we first observe that, 
   on account of \eqref{eq:limitvkCaseA}-iii), we have
   \begin{equation} \label{eq:phiketophik}
    \text{$\phi_{k,\nu}\to\phi_\nu$ pointwise a.e.\,in $\Omega$ as $k\to+\infty$};
   \end{equation}
   moreover, by the very definition of $\phi_{k,\nu}$ we also have the following estimate
   $$|\phi_{k,\nu}| =  (u_\lambda-\nu\varphi-v_k)\cdot
    \mathbf{1}_{\{u_\lambda-\nu\varphi-v_k
    \geq 0\}}\leq u_\lambda+\nu|\varphi|;$$
    as a consequence, we get
   \begin{equation} \label{eq:perfareLebesgueCaseA}
   \begin{split}
     \mathrm{i)}&\,\,v_k^{2^*-1}|\phi_{k,\nu}|
    = v_k^{2^*-1}(u_\lambda-\nu\varphi-v_k)\cdot
    \mathbf{1}_{\{u_\lambda-\nu\varphi-v_k
    \geq 0\}}\leq (u_\lambda+\nu|\varphi|)^{2^*};\\
    \mathrm{ii)}&\,\,v_k^{p}|\phi_{k,\nu}|
    = v_k^{p}(u_\lambda-\nu\varphi-v_k)\cdot
    \mathbf{1}_{\{u_\lambda-\nu\varphi-v_k
    \geq 0\}}\leq (u_\lambda+\nu|\varphi|)^{p+1}.
    \end{split}
   \end{equation}
    Using a standard dominated-convergence
    argument based on \eqref{eq:phiketophik}\,-\,\eqref{eq:perfareLebesgueCaseA}, 
    jointly with Lemma \ref{lem:Convergence} (see \eqref{eq:limitvkCaseA}\,-\,ii)
    and remind that $\varphi\in C_0^\infty(\Omega)$), we then obtain
   \begin{equation} \label{eq:limitNONLIN}
   \begin{split}
    & \lim_{k\to+\infty}
    \Big(\int_\Omega v_k^{2^*-1}(\nu\varphi+\phi_{k,\nu})\,dx 
    +\lambda\int_\Omega v_k^{p}(\nu\varphi+\phi_{k,\nu})\,dx\Big) \\
    & \qquad\qquad
    = \int_\Omega v_\lambda^{2^*-1}(\nu\varphi+\phi_{\nu})\,dx 
    + \lambda\int_\Omega v_\lambda^{p}(\nu\varphi+\phi_{\nu})\,dx.
    \end{split}
   \end{equation}
   As regards the \emph{operator term} ${\mathcal{B}_{\e}}(v_k,\nu\varphi+\phi_{k,\nu})$
   by
   proceeding \emph{exactly} as in \cite[Proposition 5.2]{BV2} (where the same operator $\LL$
   is considered), we get
   \begin{equation} \label{eq:limitOPERATORPART}
    {\mathcal{B}_{\e}}(v_k,\nu\varphi+\phi_{k,\nu}) \leq 
    {\mathcal{B}_{\e}}(v_\lambda,\nu\varphi+\phi_{\nu})+
    o(1)\qquad\text{as $k\to+\infty$},
   \end{equation}
   where we have used the fact that $v_k\to v_\lambda$ weakly in $\mathcal{X}^{1,2}(\Omega)$.
   
   Gathering \eqref{eq:limitNONLIN} and \eqref{eq:limitOPERATORPART}, and taking into account
   that ${\rho_\e}(\phi_{k,\nu})$ is \emph{uniformly bo\-un\-ded} with respect to $k$ 
   (as the same is true of $v_k$), we can finally
   pass to the limit as $k\to+\infty$ in \eqref{eq:topasstothelimitCaseASol}, obtaining
   \begin{equation} \label{eq:afterlimitkCaseA}
    \begin{split}
     {\mathcal{B}_{\e}}(v_\lambda,\nu\varphi+\phi_{\nu})
     \geq \int_\Omega v_\lambda^{2^*-1}(\nu\varphi+\phi_{\nu})\,dx 
    + \lambda\int_\Omega v_\lambda^{p}(\nu\varphi+\phi_{\nu})\,dx.
    \end{split}
   \end{equation}
   With \eqref{eq:afterlimitkCaseA} at hand, we can now exploit
   once again the computations carried out in \cite[Proposition 5.2]{BV2}, getting
   \begin{equation} \label{eq:226Haitao}
    \begin{split}
    & {\mathcal{B}_{\e}}(v_\lambda,\varphi)-
     \lambda\int_\Omega v_\lambda^{p}\varphi\,dx 
     - \int_\Omega v_\lambda^{2^*-1}\varphi\,dx  \\
     & \qquad
     \geq-\frac{1}{\nu}\Big({\mathcal{B}_{\e}}(v_\lambda,\phi_\nu)
     - \lambda\int_\Omega v_\lambda^{p}\phi_\nu\,dx 
     - \int_\Omega v_\lambda^{2^*-1}\phi_\nu\,dx\Big) \\
     & \qquad (\text{since $u_\lambda$ is a solution of \eqref{eq:Problem}$_\lambda$}) \\
     & \qquad 
     = \frac{1}{\nu}\Big(-{\mathcal{B}_{\e}}(v_\lambda-u_\lambda,\phi_\nu)
     + \lambda\int_\Omega (v_\lambda^{p}-u_\lambda^{p})\phi_\nu\,dx \\
    &\qquad\qquad + \int_\Omega (v_\lambda^{2^*-1}-u_\lambda^{2^*-1})\phi_\nu\,dx\Big) \\
    & \qquad(\text{since $v_\lambda = \textstyle\lim_{k\to+\infty}v_k\geq u_\lambda$}) \\
    & \qquad \geq \frac{1}{\nu}\Big(
    -\int_{\{v_\lambda+\nu\varphi\leq u_\lambda\}}\nabla (v_\lambda-u_\lambda)\cdot
    \nabla (v_\lambda-u_\lambda+\nu\varphi)\,dx \\
    & \qquad\qquad-\e\iint_{\R^{2n}}\frac{((v_\lambda-u_\lambda)(x)-(v_\lambda-u_\lambda)(y))
    (\phi_\nu(x)-\phi_\nu(y))}{|x-y|^{n+2s}}\,dx\,dy\Big) \\
    &\qquad \geq \text{$o(1)$ as $\nu\to 0^+$}; \phantom{\iint_{\R^{2n}}}
    \end{split}
   \end{equation}
   as a consequence, by letting $\nu\to 0^+$ in \eqref{eq:226Haitao}, we obtain
   $${\mathcal{B}_{\e}}(v_\lambda,\varphi)-
     \lambda\int_\Omega v_\lambda^{p}\varphi\,dx 
     - \int_\Omega v_\lambda^{2^*-1}\varphi\,dx\geq 0.
     $$
     This, together with the \emph{arbitrariness} of the fixed $\varphi\in C_0^\infty(\Omega)$,
     finally proves that the function
     $v_\lambda$ is a weak solution of problem \eqref{eq:Problem}$_\lambda$
     as claimed. 
     
     In particular,
     from Theorem \ref{thm:regulAntCozzi} we derive that
     \begin{equation} \label{eq:vlambdabd}
      v_\lambda\in L^\infty(\Omega).
     \end{equation}
     
     \noindent \emph{Proof of} 2). To prove assertion 2) it suffices to show that
     \begin{equation} \label{eq:claimvkstrongconv}
      \text{$v_k\to v_\lambda$ strongly in $\mathcal{X}^{1,2}(\Omega)$ as $k\to+\infty$}.
     \end{equation}
     In fact, owing to
     property a) of $\{u_k\}_k$ we have
     $$\varrho-{\rho_\e}(u_k-v_k)\leq{\rho_\e}(v_k-u_\lambda) \leq {\rho_\e}(v_k-u_k)+\varrho;$$
     this, together with \eqref{eq:claimvkstrongconv} and 
     \eqref{eq:EkelandCaseA}-ii), ensures that ${\rho_\e}(u_\lambda-v_\lambda) = \varrho.$ Hence,
     we turn to to prove \eqref{eq:claimvkstrongconv}, namely
      the \emph{strong convergence} of $\{v_k\}_k$ to $v_\lambda$. 
     \vspace*{0.1cm}
     
     First of all, by \eqref{eq:limitvkCaseA} and the 
     Brezis\,-\,Lieb Lemma \cite{BrezisLieb},  we have
     \begin{equation} \label{eq:comeLemma21Haitao}
     \begin{split}
      \mathrm{i)}&\,\,
      \text{$\|v_k-v_\lambda\|_{L^{p+1}(\Omega)}\to 0$ as $k\to+\infty$}, \\
       \mathrm{ii)}&\,\,\|v_k\|^{2^*}_{L^{2^*}(\Omega)} = \|v_\lambda\|^{2^*}_{L^{2^*}(\Omega)} + 
        \|v_k - v_\lambda\|^{2^*}_{L^{2^*}(\Omega)} +o(1); \\
        \mathrm{iii)}&\,\,{\rho_\e}(v_k)^2 = {\rho_\e}(v_\lambda)^2 + {\rho_\e}(v_k -v_\lambda)^2 + o(1).
        \end{split}
     \end{equation}
     In particular, from \eqref{eq:comeLemma21Haitao}\,-\,i) we get
     \begin{equation} \label{eq:convIntegralsallap}
      \int_\Omega v_k^{p+1}\,dx = \int_\Omega v_\lambda^{p+1}\,dx+o(1)\quad\text{as $k\to+\infty$}.
     \end{equation}
     Owing to \eqref{eq:comeLemma21Haitao}, 
     and choosing $w = v_\lambda\in T$ in \eqref{eq:221Haitao}, we then get
     \begin{equation*}
   \begin{split}
    & {\rho_\e}(v_k-v_\lambda)^2 
   = -{\mathcal{B}_{\e}}(v_k,v_\lambda-v_k)+{\mathcal{B}_{\e}}(v_\lambda,v_\lambda-v_k) \\
   & \qquad\leq \frac{1}{k}{\rho_\e}(v_k-v_\lambda)
    + \lambda\int_\Omega v_k^{p}(v_k-v_\lambda)\,dx
    \\
    & \qquad\qquad + \int_\Omega v_k^{2^*-1}(v_k-v_\lambda)\,dx
    + {\mathcal{B}_{\e}}(v_\lambda,v_\lambda-v_k) \\
    & \qquad (\text{since $\{v_k\}_k$ is bounded and $v_k\to v_\lambda$ weakly in $\mathcal{X}^{1,2}(\Omega)$}) \\
    & \qquad= \lambda\int_\Omega v_k^{p+1}\,dx +
    \int_\Omega v_k^{2^*-1}(v_k-v_\lambda)\,dx - 
    \lambda\int_\Omega v_k^{p}v_\lambda\,dx + o(1) \\
    & \qquad
    = \lambda\int_\Omega v_\lambda^{p+1}\,dx +
        \|v_k - v_\lambda\|^{2^*}_{L^{2^*}(\Omega)}
        +\|v_\lambda\|^{2^*}_{L^{2^*}(\Omega)}-\int_\Omega v_k^{2^*-1}v_\lambda\,dx
    \\
    &\qquad\qquad - 
    \lambda\int_\Omega v_k^{p}v_\lambda\,dx + o(1) \quad\text{as $k\to+\infty$}.
   \end{split}
   \end{equation*}
   On the other hand, owing to 
   \eqref{eq:limitvkCaseA}\,-\,\eqref{eq:vlambdabd}
   (and reminding that $0<p<1$), we can exploit
   once again Lemma \ref{lem:Convergence}, thus deriving
   \begin{align*}
    \int_\Omega v_k^{p}v_\lambda\,dx \to \|v_\lambda\|^{p+1}_{L^{p+1}(\Omega)}
    \qquad\text{and}\qquad
    \int_\Omega v_k^{2^*-1}v_\lambda\,dx \to \|v_\lambda\|^{2^*}_{L^{2^*}(\Omega)}
   \end{align*}
   as $k\to+\infty$; as a consequence, we obtain
   \begin{equation} \label{eq:228Haitao}
     {\rho_\e}(v_k-v_\lambda)^2\leq \|v_k-v_\lambda\|^{2^*}_{L^{2^*}(\Omega)}+o(1)\quad\text{as $k\to+\infty$}.
   \end{equation}
   To proceed further, we now choose $w = 2v_k\in T$ in \eqref{eq:221Haitao}: this yields
   \begin{equation*}
    {\rho_\e}(v_k)^2-\|v_k\|^{2^*}_{L^{2^*}(\Omega)}-\lambda\int_\Omega v_k^{p+1}\,dx\geq -
    \frac{1}{k}{\rho_\e}(v_k)^2 = o(1);
   \end{equation*}
   thus, recalling that $v_\lambda$ is a weak solution of problem \eqref{eq:Problem}$_\lambda$,
   we get
   \begin{equation} \label{eq:231Haitao}
    \begin{split}
     {\rho_\e}(v_k-v_\lambda)^2 & =
     {\rho_\e}(v_k)^2-{\rho_\e}(v_\lambda)^2+o(1)  \\
     & \geq \Big(\|v_k\|^{2^*}_{L^{2^*}(\Omega)}+\lambda\int_\Omega v_k^{p+1}\,dx\Big)
     - {\mathcal{B}_{\e}}(v_\lambda,v_\lambda) \\
     & = \|v_k\|^{2^*}_{L^{2^*}(\Omega)}+\lambda\int_\Omega v_k^{p+1}\,dx
     - \|v_\lambda\|^{2^*}_{L^{2^*}(\Omega)}-\lambda\int_\Omega v_\lambda^{p+1}\,dx \\
     & = \|v_k-v_\lambda\|^{2^*}_{L^{2^*}(\Omega)}+o(1)\quad\text{as $k\to+\infty$},
    \end{split}
   \end{equation}
   where we have also used \eqref{eq:convIntegralsallap}. Gathering 
   \eqref{eq:228Haitao}-\eqref{eq:231Haitao}, we then obtain
   \begin{equation} \label{eq:232Haitao}
    {\rho_\e}(v_k-v_\lambda)^2 = \|v_k-v_\lambda\|^{2^*}_{L^{2^*}(\Omega)}+o(1)\quad\text{as $k\to+\infty$}.
   \end{equation}
   With \eqref{eq:232Haitao} at hand, we can finally end the proof
   of \eqref{eq:claimvkstrongconv}. In fact, assuming (to fix the ideas) that $I_\lambda(u_\lambda)
   \leq I_\lambda(v_\lambda)$, from \eqref{eq:EkelandCaseA} and 
   \eqref{eq:comeLemma21Haitao}\,-\,i) we get
   \begin{align*}
    I_\lambda(v_k-v_\lambda) & = \frac{1}{2}{\rho_\e}(v_k-v_\lambda)^2
    + \frac{\lambda}{p+1}\int_\Omega|v_k-v_\lambda|^{p+1}\,dx
    + \frac{1}{2^*}\|v_k-v_\lambda\|^{2^*}_{L^{2^*}(\Omega)} \\
    & = I(v_k)-I(v_\lambda)+o(1) \leq I(u_\lambda)-I(v_\lambda)+\frac{1}{k^2}+o(1) \\
    & \leq o(1)\qquad\text{as $k\to+\infty$};
   \end{align*}
   this, together with \eqref{eq:comeLemma21Haitao}\,-\,i), gives
   \begin{equation} \label{eq:233Haitao}
   \begin{split}
    & \frac{1}{2}{\rho_\e}(v_k-v_\lambda)^2
    -\frac{1}{2^*}\|v_k-v_\lambda\|^{2^*}_{L^{2^*}(\Omega)} \\
    & \qquad
     = I_\lambda(v_k-v_\lambda)+ \frac{\lambda}{p+1}\int_\Omega|v_k-v_\lambda|^{p+1}\,dx
     \leq o(1).
    \end{split}
   \end{equation}
   Thus, by combining \eqref{eq:232Haitao}-\eqref{eq:233Haitao}, we easily obtain
   $$\lim_{k\to+\infty}\|v_k-v_\lambda\|^{2^*}_{L^{2^*}(\Omega)} = \lim_{k\to+\infty}
   {\rho_\e}(v_k-v_\lambda)^2  =0,$$
   and this proves \eqref{eq:claimvkstrongconv}.
 \end{proof}
 \begin{proposition} \label{prop:Lemma27Haitao}
  Assume that 
  \textsc{Case B)} holds. Then, there exists
  a second solution $v_\lambda$ of problem \eqref{eq:Problem}$_\lambda$ such that
  $v_\lambda\not\equiv u_\lambda$.
 \end{proposition}
 \begin{proof}
  To begin with, we consider the set
  $$\Gamma = \big\{\eta\in C([0,1];T):\,\text{$\eta(0) = u_\lambda,\,I_\lambda(\eta(1))<I_\lambda(u_\lambda)$
  and ${\rho_\e}(\eta(1)-u_\lambda) > \varrho_1$}\big\},$$
  (where $\varrho_1 > 0$ is as in \eqref{eq:defvarrho}), and we claim that $\Gamma\neq\varnothing$.
  
  In fact, since the fixed $\e$ satisfies $0<\e<\e_0$ (with $\e_0 > 0$ as
  in Lemma \ref{lem:CrucialLemma}), 
  by the cited Lemma \ref{lem:CrucialLemma} we know that there exists $R_0 > 0$ such that
   \begin{equation} \label{eq:claimHaitao}
    \begin{cases}
     I_\lambda(u_\lambda+RU_\e) < I_\lambda(u_\lambda) & \text{for all $R\geq R_0$}, \\
      I_\lambda(u_\lambda+tR_0 U_\e) < I_\lambda(u_\lambda)+\frac{1}{n}S_n^{n/2}
      & \text{for all $t\in[0,1]$}.
    \end{cases}
   \end{equation}
   In particular, from \eqref{eq:claimHaitao} we easily see that
   $$\eta_0(t) = u_\lambda+tR_0U_\e\in \Gamma$$
   (by enlarging $R_0$ if needed),
   and thus $\Gamma\neq \varnothing$, as claimed.
   \vspace*{0.1cm}
   
   Now we have proved that $\Gamma\neq\varnothing$, we can proceed towards the end of the proof. 
   To this end we first observe that, since it is non-empty, this set $\Gamma$ is a
   \emph{complete metric space}, when endowed with the distance
   $$d_\Gamma(\eta_1,\eta_2) := \max_{0\leq t\leq 1}{\rho_\e}\big(\eta_1(t)-\eta_2(t)\big);$$
   moreover, since $I_\lambda$ is \emph{real-valued and continuous} on $\mathcal{X}^{1,2}(\Omega)$,
   it is easy to recognize that the functional $\Phi:\Gamma\to \R$ defined as
   $$\Phi(\eta) := \max_{0\leq t\leq 1}I_\lambda(\eta(t)),$$
   is (well-defined and) continuous on $\Gamma$. In view of these facts,
   we are then entitled to apply the Ekeland Variational Principle
   to this functional $\Phi$ on $\Gamma$: setting
   $$\gamma_0 := \inf{\Gamma}\Phi(\eta),$$
   there exists a sequence $\{\eta_k\}_k\subseteq\Gamma$ such that
    \begin{equation} \label{eq:EkelandCaseB}
  \begin{split}
    \mathrm{i)}&\,\,\Phi(\eta_k)\leq \gamma_0+1/k, \\
    \mathrm{ii)}&\,\,\Phi(\eta_k)\leq \Phi(\eta)+1/k\,d_\Gamma(\eta_k,\eta) \quad\text{for every $\eta\in\Gamma$}.    
    \end{split}
  \end{equation}
  Now, starting from
  \eqref{eq:EkelandCaseB} and proceeding exactly as in the proof of \cite[Lemma 3.5]{BadTar}, 
  we can find another sequence
  $$v_k = \eta_k(t_k)\in T$$
  (for some $t_k\in[0,1]$) such that
  \vspace*{0.1cm}
  
    a)\,\,$I_\lambda(v_k)\to\gamma_0$ as $k\to+\infty$;
    \vspace*{0.05cm}
    
    b)\,\,there exists some $C > 0$ such that, for every $w\in T$, one has
    \begin{equation} \label{eq:237Haitao}
    \begin{split}
     & {\mathcal{B}_{\e}}(v_k,w-v_k)
      - \lambda\int_\Omega v_k^{p}(w-v_k)\,dx \\
      &\qquad\qquad 
      -\int_\Omega v_k^{2^*-1}(w-v_k)\,dx \geq -\frac{C}{k}(1+{\rho_\e}(w)).
     \end{split}
    \end{equation}
    In particular, choosing $w = 2v_k$ in \eqref{eq:237Haitao}, we get
    \begin{equation} \label{eq:choice2vk}
     {\rho_\e}(v_k)^2-\lambda\int_\Omega v^{p+1}\,dx
     -\int_\Omega v_k^{2^*}\,dx \geq -\frac{C}{k}(1+2{\rho_\e}(v_k)).
    \end{equation}
    By combining \eqref{eq:choice2vk} 
    with assertion a), and exploiting H\"older's and Sobolev's i\-ne\-qua\-lities,
    we then obtain the following estimate
    \begin{equation} \label{eq:dadedurrevkbounded}
     \begin{split}
     \gamma_0+o(1) & = \frac{1}{2}{\rho_\e}(v_k)^2-
     \frac{\lambda}{p+1}\int_\Omega v_k^{p+1}\,dx
     - \frac{1}{2^*}\int_\Omega v_k^{2^*}\,dx
     \\
     & \geq \Big(\frac{1}{2}-\frac{1}{2^*}\Big){\rho_\e}(v_k)^2
     - \lambda\Big(\frac{1}{p+1}-\frac{1}{2^*}\Big)\int_\Omega v_k^{p+1}\,dx
     \\
     & \qquad -\frac{C}{2^*\,k}(1+2{\rho_\e}(v_k)) \\
     & \geq \Big(\frac{1}{2}-\frac{1}{2^*}\Big)
      {\rho_\e}(v_k)^2- C\big({\rho_\e}(v_k)^{p+1}-2{\rho_\e}(v_k)-1\big),
     \end{split}
    \end{equation}
    where $C > 0$ is a constant depending on $n$ and on $|\Omega|$. Since, obviously,
    $$c_0 = \frac{1}{2}-\frac{1}{2^*} > 0,$$
    it is readily seen from \eqref{eq:dadedurrevkbounded} that the sequence $\{v_k\}_k$
    is \emph{bounded in $\mathcal{X}^{1,2}(\Omega)$} (otherwise, by possibly choosing a sub-sequence
    we would have ${\rho_\e}(v_k)\to+\infty$, and hence
    the right-hand side of \eqref{eq:dadedurrevkbounded} would diverges as $k\to+\infty$, which
    is not possible). 
    
    In view of this fact, we can thus proceed as in the proof of Lemma \ref{prop:Lemma26Haitao} to
    show that $\{v_k\}_k$ weakly converges (up to a sub-sequence)
    to a \emph{weak solution} $v_\lambda\in\mathcal{X}^{1,2}(\Omega)$ of 
    problem \eqref{eq:Problem}$_\lambda$,
    further satisfying the identity
    \begin{equation} \label{eq:238Haitao}
     {\rho_\e}(v_k-v_\lambda)^2-\|v_k-v_\lambda\|^{2^*}_{L^{2^*}(\Omega)} = o(1)\quad
     \text{as $k\to+\infty$}.
    \end{equation}
    In view of these facts, to complete the proof we are left to show that
    $v_\lambda\not\equiv u_\lambda$. To this end we first observe that,
    given any $\eta\in\Gamma$, we have
    $$\text{${\rho_\e}(\eta(0)-u_\lambda) = 0$\quad and \quad ${\rho_\e}(\eta(1)-u_\lambda) > \varrho_1$},$$
    and hence there exists a point $t_\eta\in[0,1]$ such that ${\rho_\e}(\eta(t_\eta)-u_\lambda) = \varrho_1$;
    as a con\-se\-quence, since \emph{we are assuming that} \textsc{Case B)} holds, we obtain
    \begin{align*}
     \gamma_0 & = \inf_{\eta\in\Gamma}\Phi(\eta)
     \geq \inf\big\{I_\lambda(\eta(t_\eta)):\,\eta\in\Gamma\big\} \\
     & \geq \inf\{I_\lambda(u):
     \text{$u\in T$ and ${\rho_\e}(u-u_\lambda) = \varrho_1$}\} > I_\lambda(u_\lambda).
    \end{align*}
    On the other hand, since we already know that $\eta_0(t) = u_\lambda+tR_0U_\e\in \Gamma$,
    from \eqref{eq:claimHaitao} (and the very definition of $\gamma_0$) we derive the following
    estimate
    \begin{align*}
     \gamma_0 & \leq \Phi(\eta_0) = \max_{0\leq t\leq 1}I_\lambda(\eta_0(t)) <
     I_\lambda(u_\lambda)+\frac{1}{n}S_n^{n/2}.
    \end{align*}
    Summing up, we have
    \begin{equation} \label{eq:236Haitao}
     I_\lambda(u_\lambda) < \gamma_0 < I_\lambda(u_\lambda)+\frac{1}{n}S_n^{n/2}.
    \end{equation}
    Now, since the sequence $\{v_k\}_k$ weakly converges in $\mathcal{X}^{1,2}(\Omega)$
    to $v_\lambda$ as $k\to+\infty$,
    \emph{the same as\-ser\-tions} in \eqref{eq:comeLemma21Haitao}
    hold also in this context; this, together with \eqref{eq:236Haitao}
    and the above property a) of the sequence $\{v_k\}_k$, gives
    \begin{equation} \label{eq:239Haitao}
     \begin{split}
      & \frac{1}{2}{\rho_\e}(v_k-v_\lambda)^2
    -\frac{1}{2^*}\|v_k-v_\lambda\|^{2^*}_{L^{2^*}(\Omega)} \\
    & \qquad
     =\frac{1}{2}\big({\rho_\e}(v_k)^2-{\rho_\e}(v_\lambda)^2\big)
      -\frac{1}{2^*}\big(\|v_k\|^{2^*}_{L^{2^*}(\Omega)}
      - \|v\|^{2^*}_{L^{2^*}(\Omega)}\big)+o(1)
     \\[0.15cm]
     & \qquad 
     = I_\lambda(v_k)-I_\lambda(u_\lambda)+o(1) 
     = \gamma_0-I_\lambda(u_\lambda)+o(1) \\
     & \qquad < \frac{1}{n}S_n^{n/2}-\delta_0,
     \end{split}
    \end{equation}
    for some $\delta_0 > 0$ such that $1/n\,S_n^{n/2}-\delta_0 > 0$ (provided that $k$ is large enough).
    
   Gathering \eqref{eq:238Haitao}-\eqref{eq:239Haitao}, 
   and arguing as in \cite[Proposition 3.1]{Tarantello}, it is then easy to recognize that
    $v_k\to v_\lambda$ \emph{strongly in $\mathcal{X}^{1,2}(\Omega)$};
    as a consequence, by the continuity of the functional
    $I_\lambda$ and by \eqref{eq:237Haitao}-\eqref{eq:236Haitao}, we get
    $$I_\lambda(u_\lambda) < \gamma_0 =  \lim_{k\to+\infty}I_\lambda(v_k) = I_\lambda(v_\lambda),$$
    and this finally proves that $v_\lambda\not\equiv u_\lambda$, as desired.
    \end{proof}
\section{Proof of Theorem \ref{thm:MAIN3}} \label{sec:Last}
As anticipated in the Introduction, in this last section
we exploit Theorem \ref{thm:main} to prove Theorem \ref{thm:MAIN3}. The 
main reason why we
postponed this proof at the very end of the paper is \emph{philosophical}: since
Theorem \ref{thm:MAIN3} concerns the behavior of the $L^\infty$-norm of any solution
of problem \eqref{eq:Problem}$_\lambda$ \emph{different from its minimal solution} $\bar{u}_{\lambda,\e}$,
we need to know that there exists (at least) \emph{one bounded solution of
\eqref{eq:Problem}$_\lambda$ distinct from such $\bar{u}_{\lambda,\e}$}; this is precisely
the content of our Theo\-rem \ref{thm:main2} (at least for $\e \ll 1$), taking into account
Theorem \ref{thm:regulAntCozzi}.
\begin{proof}[Proof (of Theorem \ref{thm:MAIN3})]
We closely follow the approach in \cite[Theorem 2.4]{ABC}. 

Arguing by contradiction, 
we assume that there exist a sequence $\{\lambda_j\}_j\subseteq (0,\Lambda_\e)$
and a family $\{v_j = v_{\lambda_j,\e}\}_j$ of weak solutions
of problem \eqref{eq:Problem}$_{\lambda_j}$ such that
\begin{itemize}
 \item[a)] $\lambda_j\to 0$ as $j\to+\infty$;
 \item[b)] $v_j\not\equiv \bar{u}_{\lambda_j,\e}$ (where $\bar{u}_{\lambda_j,\e}$
 is the unique minimal solution of \eqref{eq:Problem}$_{\lambda_j}$), and 
 \begin{equation} \label{eq:Linfbound}
  \|v_j\|_{L^\infty(\Omega)}\leq \mathbf{c}.
 \end{equation}
 for some constant $\mathbf{c} > 0$ (possibly depending on $\e$). 
\end{itemize} 
In particular, we have
\begin{equation} \label{eq:regulvj}
\begin{split}
  \bullet)\,\,&\text{$v_j > 0$ pointwise in $\Omega$ and $u = 0$ pointwise in $\de\Omega$}; \\
  \bullet)\,\,&\text{$\de_\nu v_j< 0$ pointwise on $\de\Omega$}
\end{split}
\end{equation}
 We then observe that, since $v_j$ is a weak solution of
 problem \eqref{eq:Problem}$_{\lambda_j}$ different from its mi\-ni\-mal solution
  $\bar{u}_{\lambda_j,\e}$
  (that is, problem by \eqref{eq:Problem}$_{\lambda_j}$ possesses at least two
 distinct solutions), by
 Theorem \ref{thm:main} we necessarily have that
 \begin{equation} \label{eq:vjLinfgeq}
  \|v_j\|_{L^\infty(\Omega)}\geq M_\e,
 \end{equation}
 for some constant $M_\e > 0$ (independent of $j$, but possibly depending on $\e$).
 
 On the other hand, by combining \eqref{eq:Linfbound} with the
 global regularity result proved in \cite[Theorem 1.3]{SVWZ2}
 (see also the proof of Theorem \ref{thm:regulAntCozzi}), we derive that
 \begin{equation} \label{eq:C1alfaestimvj}
 \|v_j\|_{C^{1,\alpha}(\overline{\Omega})}\leq c\,\big(1+\|v_j\|^{2^*}_{L^\infty(\Omega)}\big)
 \leq c_1,
 \end{equation}
 and such an estimate holds for every fixed $\alpha\in (0,\min\{1,2-2s\})$, 
 with 
 a suitable co\-nstant
 $c > 0$ independent of $j$. 
 Now, owing to \eqref{eq:C1alfaestimvj}
 (and bearing in mind
 \eqref{eq:regulvj}), we easily derive that
 the sequence $\{v_j\}_j$ \emph{uniformly converges} (up to a sub-sequence)
 as $j\to+\infty$ to some \emph{non-negative function}
 $$v_0\in \mathcal{X}^{1,2}(\Omega)\cap C^1(\overline{\Omega}),$$ 
 which
 is a weak solution of the \emph{purely critical problem}
 $$(\star)\qquad\quad\begin{cases}
 \LL u = |u|^{2^*-1} & \text{in $\Omega$} \\
 u = 0 &\text{in $\R^n\setminus\Omega$}
 \end{cases}$$
 (recall that $\lambda_j\to 0$ as $j\to+\infty$); moreover, by 
 \eqref{eq:C1alfaestimvj} we get
 $$\|v_0\|_{L^\infty(\Omega)}\geq M_\e\,\,\Longrightarrow\,\,v_0\not\equiv 0.$$
 Summing up, the function $v_0$ is a \emph{non-trivial and non-negative weak solution
 of $(\star)$}
 (in the sense of Definition \ref{def:weaksol}), but this is in contradiction
 with \cite[Theorem 1.3]{BDVV5} (since, by assumption, $\Omega$ is star-shaped).
 This ends the proof.
\end{proof}


\begin{thebibliography}{99} 
\bibitem{ABC}
A. Ambrosetti, H. Brezis, G. Cerami,
{\em Combined effects of concave and convex nonlinearities in some elliptic problems},
J. Funct. Anal. {\bf 122}(2), (1994), 519--543.

\bibitem{AMT}
D. Amundsen, A. Moameni, R.Y. Temgoua,
{\em Mixed local and nonlocal supercritical Dirichlet},
Commun. Pure Appl. Anal. {\bf 22}(10), (2023), 3139--3164.
 
  
\bibitem{AC2}
 C.A. Antonini, M. Cozzi, 
 {\em Global gradient regularity and a Hopf lemma for quasilinear operators of mixed local-nonlocal type}, preprint.
  \url{https://arxiv.org/abs/2308.06075}
 
 \bibitem{ArRa}
R. Arora, V. Radulescu,
\emph{Combined effects in mixed local-nonlocal stationary problems}, to appear in Proc. Roy. Soc. Edinburgh Sect. A, (2023). \, doi:10.1017/prm.2023.80

\bibitem{AubEke}
J.-P. Aubin, I. Ekeland, 
{\em Applied nonlinear analysis}.
Pure Appl. Math. (N.Y.) Wiley-Intersci. Publ.
John Wiley \& Sons, Inc., New York, 1984. xi+518 pp.

\bibitem{BadTar}
M. Badiale, G. Tarantello, 
{\em Existence and multiplicity results for elliptic problems with critical growth and discontinuous nonlinearities}, 
Nonlinear Anal. Theory Methods Appl. {\bf 29}, (1997), 639--677. 

\bibitem{BESS}
B. Barrios, E. Colorado, R. Servadei, F. Soria, 
{\em A critical fractional equation with concave-convex
power nonlinearities}, 
Ann. Inst. H.Poincar\'{e} Anal. Non Lin\'{e}aire {\bf 32}(4), (2015), 875--900.

\bibitem{BDVV}
S. Biagi, S. Dipierro, E. Valdinoci, E. Vecchi,
{\em Mixed local and nonlocal elliptic operators: regularity and maximum principles}, 
Comm. Partial Differential Equations {\bf 47} (3) (2022), 585--629.


\bibitem{BDVV3}
S. Biagi, S. Dipierro, E. Valdinoci, E. Vecchi,
{\em A Faber-Krahn inequality for mixed local and nonlocal operators}, J. Anal. Math. {\bf 150}(2), (2023), 405--448.
 

\bibitem{BDVV5}
S. Biagi, S. Dipierro, E. Valdinoci, E. Vecchi,
{\em A Brezis-Nirenberg type result for mixed local and nonlocal operators}, preprint. \url{https://arxiv.org/abs/2209.07502}

\bibitem{BEMV}
S. Biagi, F. Esposito, L. Montoro, E. Vecchi,
{\em On mixed local-nonlocal problems with Hardy potential}, forthcoming.

\bibitem{BMV}
S. Biagi, D. Mugnai, E. Vecchi,
{\em A Brezis-Oswald approach for mixed local and nonlocal operators},
Commun. Contemp. Math. {\bf 26}(2), (2024), 2250057, 28 pp.


\bibitem{BV2}
S. Biagi, E. Vecchi,
{\em Multiplicity of positive solutions for mixed local-nonlocal singular critical problems},
preprint. \url{https://arxiv.org/abs/2308.09794}

\bibitem{BCP}
J.-M. Bony, P. Courr\`ege, P. Priouret, 
{\em Semi-groupes de Feller sur une vari\'et\'e \`a bord compacte et probl\`emes aux limites 
int\'egro-diff\'erentiels du second ordre donnant lieu au principe du maximum},
Ann. Inst. Fourier (Grenoble) {\bf 18}(2), (1968), 369--521.


\bibitem{BrezisLieb}
H. Brezis, E. Lieb, 
{\em A relations between pointwise convergence of functions and convergence of integrals}, 
Proc. Amer. Math. Soc. {\bf 88}, (1983), 486--490.

\bibitem{BN}
H. Brezis, L. Nirenberg, 
{\em Positive solutions of nonlinear elliptic equation involving the critical Sobolev exponent}, 
Comm. Pure Appl. Math. {\bf 36}, (1983), 437--477. 

\bibitem{BN89}
H. Brezis, L. Nirenberg, 
{\em A minimization problem with critical exponent and nonzero data}, 
in Symmetry in Nature (a volume in honor of L. Radicati), Scuola Normale Superiore Pisa, (1989), Volume I, 129--140.

\bibitem{BNH1C1}
H. Brezis, L. Nirenberg, 
{\em $H^1$ versus $C^1$ local minimizers}, 
C. R. Acad. Sci. Paris {\bf 317}, (1993), 465--472. 

\bibitem{BO}
H. Brezis, L. Oswald, 
{\em Remarks on sublinear elliptic equations},  
Nonlinear Anal. {\bf 10} (1986), 55--64.

\bibitem{CDV22}
X. Cabr\'e, S. Dipierro, E. Valdinoci, 
{\em The Bernstein Technique for Integro-Differential
Equations}, Arch. Rational Mech. Anal. {\bf 243}, (2022), 1597--1652.

\bibitem{Cancelier}
C. Cancelier, 
{\em Probl\`emes aux limites pseudo-diff\'erentiels donnant lieu au principe du maximum},
Comm. Partial Differential Equations {\bf 11}(15), (1986), 1677--1726.

\bibitem{CCP}
F. Charro, E. Colorado, I. Peral, 
{\em Multiplicity of solutions to uniformly elliptic fully nonlinear equations with concave-convex right-hand side}, 
J. Differential Equations {\bf 246}(11), (2009), 4221--4248. 
 
\bibitem{CKSV}
 Z.-Q. Chen, P. Kim, R. Song, Z. Vondra\v{c}ek, 
 {\em Boundary Harnack principle for $\Delta + \Delta^{\alpha/2}$}, 
 Trans. Amer. Math. Soc. {\bf 364}(8), (2012), 4169--4205. 

\bibitem{CoPe}
E. Colorado, I. Peral, 
{\em Semilinear elliptic problems with mixed Dirichlet-Neumann boundary conditions}, 
J. Funct. Anal. {\bf 199}(2), (2003), 468--507.

\bibitem{daSiFi}
J.V. da Silva, A. Fiscella, V.A.B. Viloria,
{\em Mixed local-nonlocal quasilinear problems with critical nonlinearities}, preprint.  \url{https://arxiv.org/abs/2308.07460}


\bibitem{daSiSa}
J.V. da Silva, A.M. Salort, 
{\em A limiting problem for local/non-local p-Laplacians with concave-convex nonlinearities}, 
Z. Angew. Math. Phys.{\bf 71}(6), (2020), Paper No. 191, 27 pp.

 \bibitem{DeFMin}
 C. De Filippis, G. Mingione, 
 {\em Gradient regularity in mixed local and nonlocal problems}, Math. Ann. (2022). \url{https://doi.org/10.1007/s00208-022-02512-7}.

\bibitem{DPLV2}
S. Dipierro, E. Proietti Lippi, E. Valdinoci, 
{\em (Non)local logistic equations with Neumann conditions}, Ann. Inst. H. Poincar\'e Anal. Non 
Lin\'eaire (2022), 
DOI: 10.4171/AIHPC/57 .

\bibitem{DV}
S. Dipierro, E. Valdinoci,
{\em Description of an ecological niche for a mixed local/nonlocal dispersal: an evolution equation and a new Neumann condition arising from the superposition of Brownian and L\'{e}vy processes},
Phys. A. {\bf 575}, (2021), 126052.

\bibitem{GarainKinnunen}
P. Garain, J. Kinnunen,
{\em On the regularity theory for mixed local and nonlocal quasilinear elliptic equations}, Trans. Amer. Math. Soc. {\bf 375}(8), (2022), 5393--5423.
 
\bibitem{GarainLindgren}
P. Garain, E. Lindgren, 
{\em Higher Hölder regularity for mixed local and nonlocal degenerate elliptic equations}, Calc. Var. {\bf 62}, 67, (2023).
 
\bibitem{AzMaPe}
J. Garc\'{i}a Azorero, I. Peral, J. Manfredi, 
{\em Sobolev versus H\"{o}lder local minimizers and global multiplicity for some quasilinear elliptic equations},
Commun. Contemp. Math. {\bf 2}(3), (2000), 385--404.

\bibitem{KRS}
D. Kumar, V. Radulescu, K. Sreenadh,
{\em Singular elliptic problems with unbalanced growth and critical exponent},
Nonlinearity {\bf 33}(7), (2020), 3336--3369.

\bibitem{LeoniFract}
G. Leoni,
\emph{A First Course in Fractional Sobolev Spaces}, 
Graduate Studies in Mathematics (2023).

\bibitem{Moameni}
A. Moameni, 
{\em Critical point theory on convex subsets with applications in differential equations and analysis}, 
J. Math. Pures Appl. {\bf 141}, (2020), 266--315.

\bibitem{MS}
T. Mukherjee, L. Sharma
{\em On nonlocal problems with mixed operators and Dirichlet-Neumann mixed boundary conditions}, preprint.
\url{https://arxiv.org/abs/2311.02567}


\bibitem{SVWZ} 
X. Su, E. Valdinoci, Y. Wei, J. Zhang,
{\em Regularity results for solutions of mixed local and nonlocal elliptic equations}, Math. Z. {\bf 302}, (2022), 1855--1878.

\bibitem{SVWZ2} 
X. Su, E. Valdinoci, Y. Wei, J. Zhang,
{\em On Some Regularity Properties of Mixed Local and Nonlocal Elliptic Equations}, preprint. \url{https://papers.ssrn.com/sol3/papers.cfm?abstract_id=4617397}

\bibitem{Tarantello}
G. Tarantello, 
{\em On nonhomogeneous elliptic equations involving critical Sobolev exponent}, 
Ann. Inst. H. Poincar\'{e} Anal. Non Lin\'{e}aire {\bf 9}, (1992), 281--304.
  
 \end{thebibliography}
\end{document}